\documentclass[oneside, a4paper]{amsart}

	\usepackage[a4paper, margin=1in]{geometry}
	\usepackage{amsmath, amsthm, amssymb, latexsym, mathtools, mathrsfs, eucal}
	\usepackage[all,cmtip]{xy}
	\usepackage{dsfont, soul, stackrel, tikz-cd, graphicx, etoolbox, multirow}
	\DeclareGraphicsExtensions{.pdf,.png,.jpg}
	\usepackage{fancyhdr}
	\usepackage[shortlabels]{enumitem}
	\usepackage[pdfmenubar=true, pdfborder={0 0 0 [3 3]}, colorlinks, citecolor=magenta, urlcolor=blue]{hyperref}
	\numberwithin{equation}{section}
	\setenumerate{listparindent=\parindent}

	\newtheoremstyle{Mytheorem}%
	{1em}{1em}%
	{\slshape}{}%
	{\bfseries}{.}%
	{ }{}

	\newtheoremstyle{Mydefinition}%
	{1em}{1em}%
	{}{}%
	{\bfseries}{.}%
	{ }{}

	\theoremstyle{Mydefinition}
	\newtheorem{statement}{Statement}[section]
	\newtheorem{definition}[statement]{Definition}
	
	\newtheorem{remark}[statement]{Remark}
	
	\newtheorem{example}[statement]{Example}

	\newtheorem*{comment*}{Comment}
	\newtheorem{notation}[statement]{Notation}

	\theoremstyle{Mytheorem}
	\newtheorem{theorem}[statement]{Theorem}
	\newtheorem{corollary}[statement]{Corollary}
	
	\newtheorem{proposition}[statement]{Proposition}
	\newtheorem{lemma}[statement]{Lemma}

	\newcommand{\G}{GL_2^{+}(\mathbb{Q})}

	\newcommand{\pr}{\operatorname{prim}}

	\newcommand{\nc}{\newcommand}
	\newcommand{\be}{\begin{eqnarray*}}
	\newcommand{\ee}{\end{eqnarray*}}
	\newcommand{\bea}{\begin{eqnarray}}
	\newcommand{\eea}{\end{eqnarray}}
	\newcommand{\bs}{\begin{split}}
	\newcommand{\es}{\end{split}}
	\newcommand{\bal}{\begin{align}}
	\newcommand{\eal}{\end{align}}

	\nc{\bei}{\begin{itemize}}
	\nc{\eei}{\end{itemize}}
	\nc{\bee}{\begin{enumerate}}
	\nc{\eee}{\end{enumerate}}
	\nc{\bet}{\begin{thm}}
	\nc{\eet}{\end{thm}}
	\nc{\bed}{\begin{defn}}
	\nc{\eed}{\end{defn}}
	\nc{\bel}{\begin{lem}}
	\nc{\eel}{\end{lem}}
	\nc{\bep}{\begin{prop}}
	\nc{\eep}{\end{prop}}
	\nc{\bec}{\begin{corollary}}
	\nc{\eec}{\end{corollary}}
	\nc{\ber}{\begin{rem}}
	\nc{\eer}{\end{rem}}
	\nc{\beex}{\begin{example}}
	\nc{\eeex}{\end{example}}
	\nc{\bpm}{\begin{pmatrix}}
	\nc{\epm}{\end{pmatrix}}
	\nc{\bspm}{\left(\begin{smallmatrix}}
	\nc{\espm}{\end{smallmatrix}\right)}

	\newcommand{\cA}{\mathcal{A}}

	\newcommand{\cF}{\mathcal{F}}
	
	\newcommand{\cH}{\mathcal{H}}

	\newcommand{\cP}{\mathcal{P}}

	\newcommand{\cT}{\mathcal{T}}

	\newcommand{\bH}{\mathbb{H}}

	\newcommand{\BP}{\mathbf{P}}

	\nc{\frf}{\mathfrak{f}}

	\nc{\frs}{\mathfrak{s}}  
	\nc{\frt}{\mathfrak{t}} 
	\nc{\fru}{\mathfrak{u}}
	\nc{\lsl}{\mathfrak{sl}}
	\nc{\lgl}{\mathfrak{gl}}
	\nc{\upsi}{\underline{\psi}}
	\nc{\uchi}{\underline{\chi}}
	 
	\DeclareMathOperator{\Spec}{Spec}
	\DeclareMathOperator{\Proj}{Proj}

	\DeclareMathOperator{\Lie}{Lie}

	\DeclareMathOperator{\coker}{coker}
	\DeclareMathOperator{\im}{im}

	\DeclareMathOperator{\Cl}{Cl}

	\newcommand{\lra}{\longrightarrow}    
	
	\nc{\surjto}{\twoheadrightarrow}
	\nc{\ts}{\times}
	\nc{\ds}{\displaystyle}
	\nc{\nd}{\noindent}  
	\nc{\ud}{\underline}
	\nc{\ov}{\overline}
	\nc{\maplra}[1]{\buildrel #1 \over \lra}
	\nc{\mapto}[1]{\buildrel #1 \over \to}
	\nc{\setb}[1]{\{  #1\}}
	\nc{\cHom}{\mathcal{H}om}

	\def\a{\alpha}
	\def\b{\beta}

	\def\g{\gamma} \def\G{\Gamma}
	
	\def\l{\lambda} 
	\def\m{\mu}
	\def\n{\nu}

	\def\C{\mathbb{C}}
	
	\def\Z{\mathbb{Z}}

	\def\wt{\hbox{\it wt}}
	\def\ch{\hbox{\it ch}}

	\newcommand{\downtriangleBase}{%
		\begin{tikzpicture}%
			\draw[line cap=round,-] (0,1.6ex) -- (1.7ex,1.6ex);%
			\draw[line cap=round,-] (0,1.6ex) -- (0.85ex,0);%
			\draw[line cap=round,-] (0.85ex,0) -- (1.7ex,1.6ex);%
		\end{tikzpicture}%
	}
	\DeclareMathOperator{\downtriangle}{\downtriangleBase}

\title[Toric Calabi-Yau complete intersections and flat $F$-manifold structures]
{Twisted de Rham complex for toric Calabi-Yau complete intersections and flat $F$-manifold structures}

\author{Jeehoon Park}
\address{Jeehoon Park: QSMS, Seoul National University, 1 Gwanak-ro, Gwanak-gu, Seoul, South Korea 08826 }
\email{jpark.math@gmail.com}

\author{Junyeong Park}
\address{Junyeong Park: Department of Mathematical Sciences, Ulsan National Institute of Science and Technology, UNIST-gil 50, Ulsan 44919, Korea}
\email{junyeongp@gmail.com}

\begin{document}
\maketitle
\begin{abstract}
We describe the primitive middle-dimensional cohomology $\bH$ of a compact simplicial toric complete intersection variety in terms of a twisted de Rham complex.
Then this enables us to construct a concrete algorithm of formal flat $F$-manifold structures on $\bH$ in the Calabi-Yau case by using the techniques of \cite{Park23}, which turn the twisted de Rham complex into a quantization dGBV (differential Gerstenhaber-Batalin-Vilkovisky) algebra and seek for an algorithmic solution to an associated \textit{weak primitive form.}

%
\end{abstract}
\tableofcontents

\section{Introduction}

In \cite{Park23}, an explicit computer algorithm for formal flat $F$-manifold structures on the primitive middle-dimensional cohomology of a smooth projective complete intersection variety was given by the first named author; we refer to \cite{Park23} for a motivation to consider the problem of constructing formal $F$-manifold structures and relevant history with references.
The goal of the current paper is to generalize such result to a quasi-smooth complete intersection variety of ample hypersurfaces in a compact simplicial toric variety. The algorithm in \cite{Park23} was based on the description of the cohomology in terms of the top-degree cohomology of a certain twisted de Rham complex.
So, for the generalization, one needs to describe the primitive middle-dimensional cohomology $\bH$ of a quasi-smooth complete intersection variety $X$ in a compact simplicial toric variety in terms of the top-degree cohomology of a certain twisted de Rham complex, which is the main subject of the current article.

\subsection{Main theorem}
Let $\mathbb{C}$ be the field of complex numbers.
Let $\BP_\Sigma$ be the complete simplicial toric variety over $\mathbb{C}$ of dimension $n$ without torus factor associated to a rational simplicial complete fan $\Sigma$. We denote by $R_\Sigma=\mathbb{C}[x_1, \ldots, x_r]$ the toric homogeneous coordinate ring (where $r =| \Sigma(1) |$) graded by the class group $\mathrm{Cl}(\mathbf{P}_\Sigma)$. 
Let $X_G\subseteq\mathbf{P}_\Sigma$ be a quasi-smooth complete intersection variety defined by the intersections of ample hypersurfaces $X_{G_1}, \ldots, X_{G_n}$ where $X_{G_i}$ is an ample hypersurface defined by the zero locus of $G_i \in R_\Sigma$ with $\deg G_i=\beta_i\in\mathrm{Cl}(\mathbf{P}_\Sigma)$.
Consider a locally free $\CMcal{O}_{\mathbf{P}_\Sigma}$-module
\begin{align*}
\CMcal{E}:=\CMcal{O}_{\mathbf{P}_\Sigma}(\beta_1)\oplus\cdots\oplus\CMcal{O}_{\mathbf{P}_\Sigma}(\beta_k)
\end{align*}
with the associated projective bundle 
\begin{align*}
\xymatrix{\pi:\mathbf{P}(\CMcal{E}):=\underline{\Proj}_{\mathbf{P}_\Sigma}\mathrm{Sym}_{\CMcal{O}_{\mathbf{P}_\Sigma}}^\bullet\CMcal{E} \ar[r] & \mathbf{P}_\Sigma}.
\end{align*}
Then $\mathbf{P}(\CMcal{E})$ is a simplicial toric variety without torus factor whose coordinate ring is 
$
A:=R_\Sigma[y_1,\cdots,y_k]
$
graded by $\mathrm{Cl}(\mathbf{P}(\CMcal{E}))$.
Now consider the following polynomial
\begin{align*}
S:=y_1G_1+\cdots+y_kG_k\in A=R_\Sigma[y]
\end{align*}
which corresponds to the section $(G_1,\cdots,G_k)\in\Gamma(\mathbf{P}_\Sigma,\CMcal{E})$, where $\CMcal{E}$ pulls back to the invertible $\CMcal{O}_{\mathbf{P}(\CMcal{E})}$-module $\CMcal{O}_{\mathbf{P}(\CMcal{E})}(1)$ along $\pi:\mathbf{P}(\CMcal{E})\rightarrow\mathbf{P}_\Sigma$. Denote by $X_S\subseteq\mathbf{P}(\CMcal{E})$ the zero locus of $S$ and we will consider the hypersurface complement $\BP(\CMcal{E})\setminus X_S$. Then 
one has an isomorphism $H^{n+k-1}(\BP_\Sigma \setminus X_{G}) \simeq 
H^{n+k-1}(\BP(\CMcal{E}) \setminus X_S)$ induced from $\pi$; see \eqref{comp}.
The relationship between $H^{n+k-1}(\BP_\Sigma \setminus X_{G})$ and $\bH=H^{n-k}_{\mathrm{prim}}(X_{G})$ is given as follows:

\begin{theorem}\cite[Proposition 3.2]{Mav} \label{Mav}
If $X_{G}=X_{G_1} \cap \cdots \cap X_{G_k}$ is a quasi-smooth complete intersection inside $\BP_\Sigma$ where each $X_{G_i}$ is ample, then there is an exact sequence of mixed Hodge structures 
\be
0 \to H^{n-k-1}(\BP_\Sigma) \xrightarrow{\cup [X_{G}]} H^{n+k-1}(\BP_\Sigma) \to H^{n+k-1}(\BP_\Sigma \setminus X_{G}) \to \bH \to 0
\ee
where $[X_{G}] \in H^{2k}(\BP_\Sigma)$ is the cohomology class of $X_{G}$.
\end{theorem}
Thus if $\cup [X_{G}]$ is an isomorphism, then we have $\bH \simeq H^{n+k-1}(\BP(\CMcal{E}) \setminus X_S)$.
In \cite{BatCox} and \cite{Mav}, $\bH$ was described by the Jacobian ideal of $S$.
Using the Jacobian ideal description of $\bH$ and some cochain level computation (Proposition \ref{TdRCayley-rho}), we were able to describe $H^{n+k-1}(\BP(\CMcal{E}) \setminus X_S)$ as the top-degree cohomology of $(\Omega_{A/\mathbb{C}}, d+ dS\wedge-)$ exactly the same as the smooth projective complete intersection case of \cite{Park23}.
\begin{theorem}\label{mainresult}
With the same assumption as Theorem \ref{Mav},
we have an isomorphism
\bea \label{wanted}
H^{r+k}(\Omega_{A/\mathbb{C}}^\bullet, d+dS \wedge-) \simeq H^{n+k-1}(\BP(\CMcal{E}) \setminus X_S)
\eea
where $r-n$ is the rank of $\Cl(\BP_\Sigma)$.
\end{theorem}

The full description of the other degree cohomologies of $(\Omega_{A/\mathbb{C}}^\bullet, d+ dS\wedge-)$ is not that easy.
The main difference from \cite{Park23} is that the polynomial ring $A$ has multiple $\Z^m$-grading 
with $m\geq 3$ coming from the Picard group of the ambient simplicial toric variety and the Cayley trick (see subsection \ref{sec2.2}); in the case of \cite{Park23}, we have $m=2$ (the ambient toric variety is the projective space).
These extra grading makes it difficult to describe the cohomology of $(\Omega_{A/\mathbb{C}}^\bullet,d+dS\wedge-)$;  when $m=2$, \cite{AS} describes not only the top degree cohomology of $(\Omega_{A/\mathbb{C}}^\bullet,d+dS\wedge-)$ but also its total cohomology completely.
In the general case $m\geq 2$, we have some limited result, namely, we provide a limited comparison result between
the total cohomology of $\BP(\CMcal{E}) \setminus X_S$ and that of $(\Omega_{A/\mathbb{C}}^\bullet,d+dS\wedge-)$: see Theorem \ref{TdRCayley-comparison} for details.
In particular, we can provide a partial strengthening of Theorem \ref{mainresult}, 
namely, if $\Sigma$ is smooth (so that $\BP_\Sigma$ is smooth) or $\mathrm{rank}(\Cl(\BP_\Sigma))=1$, then \eqref{wanted} still holds regardless of quasi-smoothness of $X_G$ and ampleness of $X_{G_i}$.

\begin{remark}
    The results in section \ref{sec2} and appendices still hold without change of their proofs, even if one replaces the field $\mathbb{C}$ by any field of characteristic $0$. Consequently, Theorem \ref{mainresult} holds for any field of characteristic $0$.
\end{remark}

\subsection{Applications}
Once we have the isomorphism \eqref{wanted}, we can apply the techniques of \cite{Park23} to prove the following theorem.

\begin{theorem}\label{app}
If $X_G$ is a Calabi-Yau quasi-smooth complete intersection of ample hypersurfaces, then there is an explicit algorithm of formal flat $F$-manifolds on $H^{n+k-1}(\BP(\CMcal{E}) \setminus X_S)$.
\end{theorem}

More concretely, the techniques of \cite{Park23} can be summarized as follows:
\begin{itemize}
\item turn the twisted de Rham complex $(\Omega_{A/\mathbb{C}}^\bullet, d+ dS\wedge-)$ into a quantization dGBV algebra $(\cA_0^\bullet(( \hbar )),  \hbar \Delta + Q_S)$.
\item introduce the notion of a weak primitive form associated to $(\cA_0^\bullet(( \hbar )),  \hbar \Delta + Q_S)$, which provides a flat $F$-manifold structure on its top-degree cohomology.
\item construct an algorithm to find a weak primitive form.
\end{itemize}

We refer to Section \ref{sec3} for details; also see Theorem \ref{concreteversion} for a more precise version of the above theorem \ref{app}. 
%

\medskip

We briefly explain the structure of the paper.
It  consists of two main parts, section \ref{sec2} and section \ref{sec3}.
Section \ref{sec2} is devoted to describe cohomologies of a complete intersection variety $X_G$ in a compact simplicial toric variety. In subsection \ref{sec2.1}, a precise set-up is given. In subsection \ref{sec2.2}, we explain the Cayley trick which transforms the study of cohomologies of $X_G$ to one of cohomologies of a hypersurface (given by a certain polynomial $S$) complement. Subsection \ref{sec2.3} deals with de Rham cohomologies of this hypersurface complement and then we relate this to cohomologies of the twisted de Rham complex and state the key theorem \ref{TdRCayley-comparison} in subsection \ref{sec2.4}. Sections \ref{sec2.5} and \ref{sec2.6} are devoted to a proof of Theorem \ref{TdRCayley-comparison}.

Section \ref{sec3} is about an application of Theorem \ref{TdRCayley-comparison} to a construction of a formal flat $F$-manifold structure on $\bH$ based on \cite{Park23}. Subsections \ref{sec3.1} and \ref{sec3.2} are about a dGBV algebra formulation of the results in section \ref{sec2}. In subsections \ref{sec3.3} and \ref{sec3.4}, we shortly review the notions of flat $F$-manifolds and some relevant concepts. We introduce a quantization dGBV algebra and the twisted Gauss-Manin connection on it in subsection \ref{sec3.5}. Then in subsection \ref{sec3.6} we recall the notion of \textit{a weak primitive form} using the first structure connection and in subsections \ref{sec3.7} and \ref{sec3.8} we provide its explicit algorithms following \cite{Park23}.

Finally, we include appendices A and B, which are technically useful in section \ref{sec2}.

\subsection{Acknowledgement}

The work of Jeehoon Park was supported by the National Research Foundation of South Korea (NRF-2021R1A2C1006696) and the National Research Foundation of Korea (NRF) grant funded by the Korea government (MSIT) (No.2020R1A5A1016126).

Junyeong Park is supported by Samsung Science and Technology Foundation under Project Number SSTF-BA2001-02.

\section{Cohomology of complete intersections in compact simplicial toric varieties}
\label{sec2}
\subsection{Basic set up} \label{sec2.1}
The basic reference for toric varieties is \cite{CLS}. Let $M$ be a free abelian group of rank $n$, and $N:=\mathrm{Hom}(M,\mathbb{Z})$ its dual. Denote $M_\mathbb{R}:=M\otimes_\mathbb{Z}\mathbb{R}$ and $N_\mathbb{R}:=N\otimes_\mathbb{Z}\mathbb{R}$. Given a rational simplicial complete fan $\Sigma$ in $N_\mathbb{R}$, denote by $\mathbf{P}_\Sigma$ the associated simplicial toric variety\footnote{The set $\mathbf{P}_\Sigma(\mathbb{C})$ of $\mathbb{C}$-points has a natural orbifold structure, i.e. it has only finite quotient singularities (see \cite[Theorem 3.1.19]{CLS}).} of dimension $n$ over $\mathbb{C}$.

Let $\Sigma(\ell)$ be the $\ell$-dimensional cone of $\Sigma$. From now on assume that $\mathbf{P}_\Sigma$ has no torus factors, i.e., the set of minimal generators of rays in $\Sigma(1)$ spans $N_\mathbb{R}$ (see \cite[Corollary 3.3.10]{CLS}). Then \cite[Theorem 4.1.3]{CLS} gives an exact sequence
\begin{align}\label{ray-cl-exact}
\begin{aligned}
\xymatrix{0 \ar[r] & M \ar[r] & \displaystyle\bigoplus_{\rho\in\Sigma(1)}\mathbb{Z}D_\rho \ar[r] & \mathrm{Cl}(\mathbf{P}_\Sigma) \ar[r] & 0}
\end{aligned}
\end{align}
where $D_\rho$ is a $\mathbf{T}_N$-invariant prime divisor\footnote{Here $\mathbf{T}_N:=\Spec\mathbb{C}[M]\cong\mathbf{G}_m^n$ is the torus associated to $N$. In fact, $D_\rho$ is the Zariski closure of $\mathcal{O}(\rho)\simeq\mathrm{Hom}(\rho^{\bot} \cap M,\mathbb{C}^*)$ where $\mathcal{O}(\rho)$ corresponds to $\rho$ under the Orbit-Cone correspondence, \cite[Theorem 3.2.6]{CLS}. Here $^{\bot}$ means the orthogonal complement under the natural pairing on $M$ and $N$.} in $\mathbf{P}_\Sigma$ \cite[section 4.1]{CLS} and $\mathrm{Cl}(\mathbf{P}_{\Sigma})$ is the class group of $\mathbf{P}_{\Sigma}$. For simplicity of presentation, we make the following additional assumption:
\begin{itemize}
\item $\mathrm{Cl}(\mathbf{P}_\Sigma)$ is a free abelian group.
\end{itemize}
Let $\mathbf{G}_\Sigma:=\Spec\mathbb{C}[\mathrm{Cl}(\mathbf{P}_\Sigma)]$. Then \eqref{ray-cl-exact} induces an exact sequence
\begin{align*}
\xymatrix{0 \ar[r] & \mathbf{G}_\Sigma \ar[r] & \mathbf{G}_m^r \ar[r] & \mathbf{G}_m^n \ar[r] & 0}.
\end{align*}
with $r=|\Sigma(1)|$. On the other hand, denote the total coordinate ring of $\mathbf{P}_\Sigma$ by
\begin{align*}
R_{\Sigma}:=\mathbb{C}[x_\rho : \rho \in \Sigma(1)]
\end{align*}
Then $\mathbf{G}_\Sigma$ acts on $\Spec R_\Sigma\cong\mathbf{A}^r$ via the componentwise action of $\mathbf{G}_m^r$ on $\mathbf{A}^r$. This action induces a $\mathrm{Cl}(\mathbf{P}_\Sigma)$-grading on $R_\Sigma$:
\begin{align*}
\deg_c\left(\prod_{\rho\in\Sigma(1)}x_\rho^{a_\rho}\right) =\left[\sum_{\rho\in\Sigma(1)}a_\rho D_\rho\right]\in\mathrm{Cl}(\mathbf{P}_{\Sigma}).
\end{align*}
For $\beta\in\mathrm{Cl}(\mathbf{P}_\Sigma)$, we denote $R_{\Sigma,\beta}\subseteq R_\Sigma$ the corresponding graded piece. For each $\sigma\in\Sigma$ denote
\begin{align*}
\widehat{x}_\sigma:=\prod_{\rho\nsubseteq\sigma}x_\rho\in R_\Sigma.
\end{align*}
Let $B_\Sigma:=\langle\widehat{x}_\sigma\ |\ \sigma\in\Sigma\rangle\subseteq R_\Sigma$ be the ideal generated by the $\widehat{x}_\sigma$ above. Then this ideal gives the affine variety $Z_\Sigma:=V(B_\Sigma)\subseteq\Spec R_\Sigma\cong\mathbf{A}^r$ with complement $U_\Sigma:=\mathbf{A}^r\setminus Z_\Sigma$. Now we have
\begin{align*}
\mathbf{P}_\Sigma\cong U_\Sigma/\mathbf{G}_\Sigma
\end{align*}
(see \cite[section 5.1]{CLS} for details in the case $\mathbb{C}=\mathbb{C}$).

Given homogeneous polynomials $G_1,\cdots,G_k\in R_\Sigma$ with $\deg_cG_i=\beta_i\in\mathrm{Cl}(\mathbf{P}_\Sigma)$, denote $X_G\subseteq\mathbf{P}_\Sigma$ their common zero locus. Later, we focus on the case where $X_G$ is a quasi-smooth complete intersection (\cite[Definition 1.1]{Mav}) of ample hypersurfaces (i.e. $\deg_cG_i=\beta_i\in\mathrm{Cl}(\mathbf{P}_\Sigma)$ is the class of an ample divisor on $\mathbf{P}_\Sigma$). In this case, $G_i \in B_\Sigma$ by \cite[Lemma 9.15]{BatCox} and $X_G$ has pure dimension $n-k$ (\cite[Corollary 1.6]{Mav}). Moreover, $X_G$ is connected provided $k < n$ (\cite[Corollary 1.7]{Mav}).

\subsection{The Cayley trick}\label{sec2.2}
We will express the algebraic de Rham cohomology $H_{\widehat{\mathrm{dR}}}^\bullet(\mathbf{P}_\Sigma\setminus X_G)$ in terms of a hypersurface complement of another simplicial toric variety. Consider a locally free $\CMcal{O}_{\mathbf{P}_\Sigma}$-module
\begin{align*}
\CMcal{E}:=\CMcal{O}_{\mathbf{P}_\Sigma}(\beta_1)\oplus\cdots\oplus\CMcal{O}_{\mathbf{P}_\Sigma}(\beta_k)
\end{align*}
with the associated projective bundle (cf. \cite[\href{https://stacks.math.columbia.edu/tag/01OA}{Tag 01OA}]{Stacks})
\begin{align*}
\xymatrix{\pi:\mathbf{P}(\CMcal{E}):=\underline{\Proj}_{\mathbf{P}_\Sigma}\mathrm{Sym}_{\CMcal{O}_{\mathbf{P}_\Sigma}}^\bullet\CMcal{E} \ar[r] & \mathbf{P}_\Sigma}.
\end{align*}
By the projective bundle formula for Chow groups (see \cite[Theorem 3.3]{WF} or \cite[\href{https://stacks.math.columbia.edu/tag/02TX}{Tag 02TX}]{Stacks}), we have
\begin{align*}
\mathrm{Cl}(\mathbf{P}(\CMcal{E}))\cong\mathrm{CH}_{n+k-2}(\mathbf{P}(\CMcal{E}))\cong\bigoplus_{i=0}^{k-1}\mathrm{CH}_{n-1+i}(\mathbf{P}_\Sigma)\cong
\left\{\begin{array}{ll}
\mathrm{Cl}(\mathbf{P}_\Sigma) & \textrm{if $k=1$} \\
\mathrm{Cl}(\mathbf{P}_\Sigma)\oplus\mathbb{Z} & \textrm{if $k>1$}
\end{array}\right.
\end{align*}
where we use $\mathrm{CH}_n(\mathbf{P}_\Sigma)\cong\mathbb{Z}$ which is generated by the class $[\mathbf{P}_\Sigma]\in\mathrm{CH}_n(\mathbf{P}_\Sigma)$. Moreover, by \cite[pp.135--136]{Mav} or \cite[pp.58--59]{TOCB}, $\mathbf{P}(\CMcal{E})$ is a simplicial toric variety without torus factor. In what follows, we enumerate $\Sigma(1)=\{\rho_1,\cdots,\rho_r\}$ and denote $x_i:=x_{\rho_i}$. By introducing $q_1:=y_1,\cdots,q_k:=y_k$ and $q_{k+1}:=x_1,\cdots,q_{r+k}:=x_r$, we will use the following notation
\begin{align}\label{ringa}
A=R_\Sigma[y]=\mathbb{C}[x_1, \ldots, x_r, y_1, \ldots, y_k]=\mathbb{C}[q].
\end{align}
Given a homogeneous $f\in A$, we use the identification $\mathrm{Cl}(\mathbf{P}(\CMcal{E}))\cong\mathrm{Cl}(\mathbf{P}_\Sigma)\oplus\mathbb{Z}$ above to write
\begin{align*}
\deg f=(\deg_cf,\deg_wf)\in\mathrm{Cl}(\mathbf{P}_\Sigma)\oplus\mathbb{Z}.
\end{align*}

\begin{notation}\label{notation-deg} Given a $(\mathrm{Cl}(\mathbf{P}_\Sigma)\oplus\mathbb{Z})$-graded $\mathbb{C}$-vector space $\Xi$, denote
\begin{align*}
\Xi_\lambda&:=\{\xi\in\Xi\ |\ \deg_c\xi=\lambda\}\\
\Xi_{(m)}&:=\{\xi\in\Xi\ |\ \deg_w\xi=m\}
\end{align*}
which are $\mathbb{C}$-linear subspaces of $\Xi$. Moreover, we sometimes denote
\begin{align*}
\Xi_{(+)}:=\{\xi\in\Xi\ |\ \deg_w\xi>0\}
\end{align*}
Double subscripts denote the corresponding intersections:
\begin{align*}
\Xi_{c,(w)}:=\Xi_c\cap\Xi_{(w)},\quad c\in\mathrm{Cl}(\mathbf{P}_\Sigma),\quad w\in\{+\}\cup\mathbb{Z}.
\end{align*}
For example, we have
\begin{align*}
\widehat{\Omega}_{A/\mathbb{C}}^\bullet&=\bigoplus_{0\leq i\leq r+k}\bigoplus_{\lambda\in\mathrm{Cl}(\mathbf{P}_\Sigma)}\bigoplus_{m\geq0}\left(\widehat{\Omega}_{A/\mathbb{C}}^i\right)_{\lambda,(m)}
\end{align*}
\end{notation}

From now on, we focus on the case $k>1$, in which case the isomorphism is given by
\begin{align*}
\xymatrix{\mathrm{Cl}(\mathbf{P}_\Sigma)\oplus\mathbb{Z} \ar[r]^-\sim & \mathrm{Cl}(\mathbf{P}(\CMcal{E})) & (\gamma,1) \ar@{|->}[r] & \pi^\ast\gamma+\left[\CMcal{O}_{\mathbf{P}(\CMcal{E})}(1)\right]}
\end{align*}
Consequently, we have
\begin{align*}
\deg G_i=(\beta_i,0)\in\mathrm{Cl}(\mathbf{P}_\Sigma)\oplus\mathbb{Z},\quad\deg y_i=(-\beta_i,1)\in\mathrm{Cl}(\mathbf{P}_\Sigma)\oplus\mathbb{Z}.
\end{align*}
Now the section
\begin{align}\label{dwork}
S:=y_1G_1+\cdots+y_kG_k\in A=R_\Sigma[y]=\mathbb{C}[q]
\end{align}
is of degree $(0,1)\in\mathrm{Cl}(\mathbf{P}_\Sigma)\oplus\mathbb{Z}$. Note that $S$ corresponds to the section $(G_1,\cdots,G_k)\in\Gamma(\mathbf{P}_\Sigma,\CMcal{E})$, where $\CMcal{E}$ pulls back to the invertible $\CMcal{O}_{\mathbf{P}(\CMcal{E})}$-module $\CMcal{O}_{\mathbf{P}(\CMcal{E})}(1)$ along $\pi:\mathbf{P}(\CMcal{E})\rightarrow\mathbf{P}_\Sigma$. Denote by $X_S\subseteq\mathbf{P}(\CMcal{E})$ the zero locus of $S$. Then we get a commutative diagram
\begin{align}\label{Cayley-diagram}
\begin{aligned}
\xymatrix{
\Spec A[S^{-1}] \ar[d] \ar[r] & U_\Sigma\times_\mathbb{C}(\mathbf{A}^k\setminus0) \ar[d] \\
\mathbf{P}(\CMcal{E})\setminus X_S \ar[d]_-\varphi \ar[r] & \mathbf{P}(\CMcal{E}) \ar[d]^-\pi \\
\mathbf{P}_\Sigma\setminus X_G \ar[r] & \mathbf{P}_\Sigma\cong U_\Sigma/\mathbf{G}_\Sigma
}
\end{aligned}
\end{align}
Denote $X_{G_i}\subseteq\mathbf{P}_\Sigma$ the hypersurface cut out by $G_i$ so $\{\mathbf{P}_\Sigma\setminus X_{G_i}\}_{i=1,\cdots,k}$ is an open covering of $\mathbf{P}_\Sigma\setminus X_G$ such that $\CMcal{E}|_{\mathbf{P}_\Sigma\setminus X_{G_i}}\cong\CMcal{O}_{\mathbf{P}_\Sigma\setminus X_{G_i}}^{\oplus k-1}$. Hence $\varphi$ is an $\mathbf{A}^{k-1}$-bundle. By Corollary \ref{dR-affine-bundle},
\begin{align}\label{comp}
\xymatrix{\varphi^\ast:H_{\widehat{\mathrm{dR}}}^\bullet(\mathbf{P}_\Sigma\setminus X_G) \ar[r] & H_{\widehat{\mathrm{dR}}}^\bullet(\mathbf{P}(\CMcal{E})\setminus X_S)}
\end{align}
is an isomorphism. Moreover,
\begin{align*}
\mathbf{P}(\CMcal{E})\cong\frac{\left(U_\Sigma\times_\mathbb{C}(\mathbf{A}^k\setminus0)\right)}{\Spec\mathbb{C}[\mathrm{Cl}(\mathbf{P}(\CMcal{E}))]}.
\end{align*}
and $\mathbf{P}(\CMcal{E})\setminus X_S$ is affine with coordinate ring
\begin{align*}
A[S^{-1}]^\mathbf{G}= A[S^{-1}]_{0,(0)}, \quad\mathbf{G}:=\mathbf{G}_\Sigma\times_\mathbb{C}\mathbf{G}_m
\end{align*}
so we get an isomorphism of cohomology spaces\footnote{see appendix \ref{algdRnormal}; we need to take the \emph{reflexive hull} $\ \widehat{\cdot} \ $ because $A[S^{-1}]^\mathbf{G}$ is not necessarily smooth.}:
\begin{align} \label{ciso}
H_{\widehat{\mathrm{dR}}}^\bullet(\mathbf{P}(\CMcal{E})\setminus X_S)\cong H^\bullet\left(\widehat{\Omega}_{A[S^{-1}]^\mathbf{G}/\mathbb{C}}^\bullet,d\right).
\end{align}

\subsection{The algebraic de Rham cohomologies of hypersurface complements}
\label{sec2.3}
 In this subsection, following the appendix \ref{actiondR}, we provide a relationship between
\begin{align*}
\left(\widehat{\Omega}_{A[S^{-1}]^\mathbf{G}/\mathbb{C}}^\bullet,d\right)\quad\textrm{and}\quad\left(\Omega_{A[S^{-1}]/\mathbb{C}}^\bullet,d\right).
\end{align*}

Given a finitely generated abelian group $V$, the group algebra $\mathbb{C}[V]$ is a Hopf algebra with respect to the following maps (cf. \cite[pp. 230--231]{MAG}):
\begin{align*}
\xymatrixrowsep{0pc}\xymatrix{
\Delta_{\mathbb{C}[V]/\mathbb{C}}:\mathbb{C}[V] \ar[r] & \mathbb{C}[V]\otimes_\mathbb{C}\mathbb{C}[V] & x^v \ar@{|->}[r] & x^v\otimes x^v \\
\epsilon_{\mathbb{C}[V]/\mathbb{C}}:\mathbb{C}[V] \ar[r] & \mathbb{C} & x^v \ar@{|->}[r] & 1 \\
S_{\mathbb{C}[V]/\mathbb{C}}:\mathbb{C}[V] \ar[r] & \mathbb{C}[V] & x^v \ar@{|->}[r] & x^{-v}
}
\end{align*}
so $G_V:=\Spec\mathbb{C}[V]$ is an affine algebraic group over $\mathbb{C}$. Denote $\mathfrak{I}_V\subseteq\mathbb{Z}[V]$ the augmentation ideal so that $\mathfrak{I}_V/\mathfrak{I}_V^2\cong V$. Hence
\begin{align*}
\mathrm{Lie}(G_V)\cong\mathrm{Hom}_\mathbb{C}\left(\mathfrak{I}_V/\mathfrak{I}_V^2\otimes_\mathbb{Z}\mathbb{C},\mathbb{C}\right)\cong\mathrm{Hom}(V,\mathbb{C})
\end{align*}
naturally in $V$ (cf. \cite[p. 188]{MAG}). Consequently, (\ref{ray-cl-exact}) gives a commutative diagram of $\mathbb{C}$-vector spaces
\begin{align}\label{Lie-dual-exact}
\begin{aligned}
\xymatrix{
0 \ar[r] & \mathrm{Lie}(\mathbf{G}_\Sigma) \ar[d]^-\wr \ar[r] & \mathrm{Lie}(\mathbf{G}_m^r) \ar[d]^-\wr \ar[r] & \mathrm{Lie}(\mathbf{G}_m^n) \ar[d]^-\wr \ar[r] & 0 \\
0 \ar[r] & \mathrm{Hom}(\mathrm{Cl}(\mathbf{P}_\Sigma),\mathbb{C}) \ar[r] & \mathrm{Hom}(\mathbb{Z}^{\oplus\Sigma(1)},\mathbb{C}) \ar[r] & \mathrm{Hom}(M,\mathbb{C}) \ar[r] & 0
}
\end{aligned}
\end{align}
with exact rows. Identify $\mathbb{Z}^{\oplus\Sigma(1)}$ with its $\mathbb{Z}$-dual by setting $\langle D_\alpha,D_\beta\rangle:=\delta_{\alpha\beta}$ (the Kronecker delta):
\begin{align}\label{ray-dual}
\xymatrix{\mathbb{Z}^{\oplus\Sigma(1)} \ar[r]^-\sim & \mathrm{Hom}(\mathbb{Z}^{\oplus\Sigma(1)},\mathbb{Z}) & D_\rho \ar@{|->}[r] & \langle D_\rho,-\rangle}.
\end{align}
Then for each $\rho\in\Sigma(1)$, there is a map
\begin{align*}
\xymatrix{\delta_\rho:R_\Sigma \ar[r] & R_\Sigma & x^u \ar@{|->}[r] & \langle D_\rho,u\rangle x^u}
\end{align*}
which is a $\mathbb{C}$-derivation:
\begin{align*}
\delta_\rho(x^{u+v})=\langle D_\rho,u+v\rangle x^{u+v}=\langle D_\rho,u\rangle x^ux^v+\langle D_\rho,v\rangle x^vx^u=\delta_\rho(x^u)x^v+\delta_\rho(x^v)x^u.
\end{align*}
which naturally extends to $\mathbb{C}[\mathbb{Z}^{\oplus\Sigma(1)}]\cong R_\Sigma[x_\rho^{-1}\ |\ \rho\in\Sigma(1)]$. Hence we get a $\mathbb{C}$-linear injection
\begin{align*}
\xymatrix{\displaystyle\bigoplus_{\rho\in\Sigma(1)}\mathbb{C} D_\rho \ar@{^(->}[r] & \mathrm{Der}_\mathbb{C}(\mathbb{C}[\mathbb{Z}^{\oplus\Sigma(1)}],\mathbb{C}[\mathbb{Z}^{\oplus\Sigma(1)}]) & D_\rho \ar@{|->}[r] & \delta_\rho}.
\end{align*}
Moreover, we have
\begin{align*}
(\Delta_{R_\Sigma/\mathbb{C}}\circ\delta_\rho)(x^u)=\langle D_\rho,u\rangle x^u\otimes x^u=x^u\otimes \langle D_\rho,u\rangle x^u=((\mathrm{Id}_{R_\Sigma}\otimes_\mathbb{C}\delta_\rho)\circ\Delta_{R_\Sigma/\mathbb{C}})(x^u)
\end{align*}
for every $u\in\mathbb{Z}^{\oplus\Sigma(1)}$ so each $\delta_\rho$ is a left invariant differential (cf. \cite[p. 195]{MAG}). Hence we get a $\mathbb{C}$-linear isomorphism up to identification of the Lie algebra of algebraic groups and left invariant derivations (cf. \cite[Proposition 10.28]{MAG} and \cite[Proposition 10.29]{MAG}):
\begin{align*}
\xymatrix{\displaystyle\bigoplus_{\rho\in\Sigma(1)}\mathbb{C} D_\rho \ar[r]^-\sim & \mathrm{Lie}(\mathbf{G}_m^r) & D_\rho \ar@{|->}[r] & \delta_\rho}.
\end{align*}
Along this identification, the left column in (\ref{Lie-dual-exact}) can be written as follows:
\begin{align*}
\xymatrix{\mathrm{Hom}(\mathrm{Cl}(\mathbf{P}_\Sigma),\mathbb{C}) \ar[r] & \mathrm{Lie}(\mathbf{G}_\Sigma) & \psi \ar@{|->}[r] & \displaystyle\delta_\psi:=\sum_{\rho\in\Sigma(1)}\psi([D_\rho])\delta_\rho}.
\end{align*}
From the projective bundle formula above and Example \ref{Euler-Gm} together with the natural isomorphism $\mathrm{Hom}(-,\mathbb{C})\cong\mathrm{Lie}(\Spec\mathbb{C}[-])$, we get the following commutative diagram:
\begin{align*}
\xymatrix{
\mathrm{Hom}(\mathrm{Cl}(\mathbf{P}_\Sigma)\oplus\mathbb{Z},\mathbb{C}) \ar[d]^-\wr \ar[r]^-\sim & \mathrm{Lie}(\mathbf{G}_\Sigma)\oplus\mathrm{Lie}(\mathbf{G}_m) \ar[d]^-\wr & (\psi,1) \ar@{|->}[r] & \displaystyle\left(\delta_\psi,\lambda\frac{\partial}{\partial\lambda}\right) \\
\mathrm{Hom}(\mathrm{Cl}(\mathbf{P}(\CMcal{E})),\mathbb{C}) \ar[r]^-\sim & \mathrm{Lie}(\mathbf{G}) & &
}
\end{align*}
Then $(\psi,0),(0,1)\in\mathrm{Hom}(\mathrm{Cl}(\mathbf{P}_\Sigma),\mathbb{C})\oplus\mathbb{C}$ give Euler vector fields
\begin{align*}
E_\psi&:=\sum_{\rho\in\Sigma(1)}\delta_{(\psi,0)}(x_\rho)\frac{\partial}{\partial x_\rho}+\sum_{j=1}^k\delta_{(\psi,0)}(y_j)\frac{\partial}{\partial y_j}=\sum_{\rho\in\Sigma(1)}\psi([D_\rho])x_\rho\frac{\partial}{\partial x_\rho}+\sum_{j=1}^k\psi(-\beta_j)y_j\frac{\partial}{\partial y_j}\\
E_w&:=\sum_{\rho\in\Sigma(1)}\delta_{(0,1)}(x_\rho)\frac{\partial}{\partial x_\rho}+\sum_{j=1}^k\delta_{(0,1)}(y_j)\frac{\partial}{\partial y_j}=\sum_{j=1}^ky_j\frac{\partial}{\partial y_j}.
\end{align*}
on $R_\Sigma$ respectively. Denote by
\begin{align*}
\theta_\psi:=\langle E_\psi,-\rangle,\quad\theta_w:=\langle E_w,-\rangle
\end{align*}
the contraction with each Euler vector field.

By Proposition \ref{torusdeg}, a set of basis for $\mathrm{Hom}(\mathrm{Cl}(\mathbf{P}_\Sigma),\mathbb{Z})$, say $\{\epsilon_1,\cdots,\epsilon_{r-n}\}$ gives a system of Euler vector fields $E_1,\cdots,E_{r-n}$ and the corresponding $\theta_{c,1},\cdots,\theta_{c,r-n}$ together with gradings $\deg_{c,1},\cdots,\deg_{c,r-n}$ so
\begin{align*}
\left(\Omega_{A[S^{-1}]/\mathbb{C}}^\bullet,d\right)^\mathbf{G}=\left(\Omega_{A[S^{-1}]/\mathbb{C}}^\bullet,d\right)_{0,(0)}
\end{align*}
Moreover, $\left(\widehat{\Omega}_{A[S^{-1}]^\mathbf{G}/\mathbb{C}}^\bullet,d\right)$ is the submodule of $\left(\Omega_{A[S^{-1}]/\mathbb{C}}^\bullet,d\right)^\mathbf{G}$ annihilated by $\theta_{c,1},\cdots,\theta_{c,r-n}$ and $\theta_w$.
\begin{lemma}\label{Cayley-theta} $\theta_{c,i}$, $1\leq i\leq r-n$, and $\theta_w$ above have the following properties.
\begin{quote}
(1) For $1\leq i,j\leq r-n$, the following identities hold:
\begin{align*}
\theta_{c,i}^2=0,\quad\theta_w^2=0,\quad\theta_{c,i}\theta_{c,j}+\theta_{c,j}\theta_{c,i}=0,\quad\theta_{c,i}\theta_w+\theta_w\theta_{c,i}=0
\end{align*}
(2) $\theta_{c,i}$ and $\theta_w$ are derivations of the wedge product, i.e. if $\alpha$ is an $\ell$-form, then
\begin{align*}
\theta_{c,i}(\alpha\wedge\beta)=\theta_{c,i}\alpha\wedge\beta+(-1)^\ell\alpha\wedge\theta_{c,i}\beta
\end{align*}
\begin{align*}
\theta_w(\alpha\wedge\beta)=\theta_w\alpha\wedge\beta+(-1)^\ell\alpha\wedge\theta_w\beta
\end{align*}
(3) For a homogeneous $f$ and $\lambda\in\mathbb{C}$, if we denote
\begin{align*}
D_{\lambda,f}:=\lambda d+df\wedge-
\end{align*}
then for a homogeneous $\xi \in \Omega_{A/\mathbb{C}}^\bullet$,
\begin{align*}
(D_{\lambda,f}\theta_{c,i}+\theta_{c,i}D_{\lambda,f})\xi=\left(\lambda\deg_{c,i}\xi+(\deg_{c,i}f)f\right)\xi
\end{align*}
\begin{align*}
(D_{\lambda,f}\theta_w+\theta_wD_{\lambda,f})\xi=\left(\lambda\deg_w\xi+(\deg_wf)f\right)\xi.
\end{align*}
Note that $\lambda\in\mathbb{C}$ are of degree zero.
\end{quote}
\end{lemma}
\begin{proof}
They follow from direct computation.
\end{proof}

\begin{proposition}\label{dRconcent} With the notations above, the following hold.
\begin{quote}
(1) The inclusions in the commutative square are quasi-isomorphisms.
\begin{align*}
\xymatrix{
\left(\Omega_{A[S^{-1}]/\mathbb{C}}^\bullet,d\right)_{0,(0)} \ar@{^(->}[d] \ar@{^(->}[r] & \left(\Omega_{A[S^{-1}]/\mathbb{C}}^\bullet,d\right)_0 \ar@{^(->}[d] \\
\left(\Omega_{A[S^{-1}]/\mathbb{C}}^\bullet,d\right)_{(0)} \ar@{^(->}[r] & \left(\Omega_{A[S^{-1}]/\mathbb{C}}^\bullet,d\right)
}
\end{align*}
(2) There are cochain maps induced from $\theta_{c,i}$ and $\theta_w$ respectively:
\begin{align*}
\xymatrix{\theta_{c,i}:\left(\Omega_{A[S^{-1}]/\mathbb{C}}^\bullet,d\right)_0 \ar[r] & \left(\Omega_{A[S^{-1}]/\mathbb{C}}^\bullet[-1],d\right)_0}
\end{align*}
\begin{align*}
\xymatrix{\theta_w:\left(\Omega_{A[S^{-1}]/\mathbb{C}}^\bullet,d\right)_{(0)} \ar[r] & \left(\Omega_{A[S^{-1}]/\mathbb{C}}^\bullet[-1],d\right)_{(0)}}
\end{align*}
\end{quote}
\end{proposition}
\begin{proof}
By Lemma \ref{Cayley-theta} (3), we get relations for homogeneous $\xi$:
\begin{align*}
(d\theta_{c,i}+\theta_{c,i}d)\xi=(\deg_{c,i}\xi)\xi,\quad(d\theta_w+\theta_wd)\xi=(\deg_w\xi)\xi
\end{align*}
\indent (1) If $d\xi=0$, then $\xi$ is in the image of $d$ unless $(\deg_c\xi,\deg_w\xi)=(0,0)$.\\
\indent (2) The above relations also show that each $\theta$ becomes a cochain map on the subcomplex of homogeneous elements of degree $0$.
\end{proof}
\begin{proposition}\label{Cayley-decomposition} The inclusion $A[S^{-1}]_{(0)}\subseteq A[S^{-1}]$ induces
\begin{align*}
\left(\widehat{\Omega}_{A[S^{-1}]_{(0)}/\mathbb{C}}^\bullet,d\right)\cong\left(\Omega_{A[S^{-1}]_{(0)}/\mathbb{C}}^\bullet,d\right)\cong\ker\left(\xymatrixcolsep{1.5pc}\xymatrix{\theta_w:\left(\Omega_{A[S^{-1}]/\mathbb{C}}^\bullet,d\right)_{(0)} \ar[r] & \left(\Omega_{A[S^{-1}]/\mathbb{C}}^\bullet[-1],d\right)_{(0)}}\right)
\end{align*}
and there is a decomposition of complexes
\begin{align*}
\left(\Omega_{A[S^{-1}]/\mathbb{C}}^\bullet,d\right)_{(0)}\cong\left(\Omega_{A[S^{-1}]_{(0)}/\mathbb{C}}^\bullet,d\right)\oplus\frac{dS}{S}\wedge\left(\Omega_{A[S^{-1}]_{(0)}/\mathbb{C}}^\bullet,d\right)
\end{align*}
\end{proposition}
\begin{proof} Since $\Spec A[S^{-1}]_{(0)}$ is an affine open subset of a smooth $\mathbb{C}$-scheme $U_\Sigma\times_\mathbb{C}\mathbf{P}^{k-1}_\mathbb{C}$, it is smooth so $\Omega_{A[S^{-1}]_{(0)}/\mathbb{C}}^\bullet$ is a complex of locally free $A[S^{-1}]_{(0)}$-modules which are reflexive and consequently $\widehat{\Omega}_{A[S^{-1}]_{(0)}/\mathbb{C}}^\bullet\cong\Omega_{A[S^{-1}]_{(0)}/\mathbb{C}}^\bullet$. The isomorphism $\Omega_{A[S^{-1}]_{(0)}/\mathbb{C}}^\bullet\cong\ker\theta_w$ follows from \eqref{exabc} and Proposition \ref{torusdeg} (5). 
From the equation for $\xi\in\Omega_{A[S^{-1}]/\mathbb{C}}^\bullet$
\begin{align}\label{thetawdecomp}
\theta_w\left(\frac{dS}{S}\wedge\xi\right)+\frac{dS}{S}\wedge\theta_w(\xi) =
\frac{\theta_w(dS)}{S} \wedge \xi - \frac{dS}{S} \wedge \theta_w(\xi) + \frac{dS}{S} \wedge \theta_w(\xi) =\xi
\end{align}
(since $\theta_w(dS)=\theta_w(dS) + d \theta_w(S) = S$),
we deduce that
\begin{align*}
\Omega_{A[S^{-1}]/\mathbb{C}}^\bullet \subseteq \ker\theta_w+\frac{dS}{S}\wedge\ker\theta_w
\end{align*}
as a graded vector space. Because the opposite inclusion clearly holds ($A[S^{-1}]_{(0)} \subset A[S^{-1}]$ and the isomorphism $\Omega_{A[S^{-1}]_{(0)}/\mathbb{C}}^\bullet\cong\ker\theta_w$), we have 
\begin{align*}
\Omega_{A[S^{-1}]/\mathbb{C}}^\bullet = \ker\theta_w+\frac{dS}{S}\wedge\ker\theta_w.
\end{align*}
Moreover, if $\xi$ is contained in the intersection of summands, i.e.
\begin{align*}
\xi\in\ker\theta_w\cap\frac{dS}{S}\wedge\ker\theta_w
\end{align*}
then we may rewrite $\xi$ as
\begin{align*}
\xi=\frac{dS}{S}\wedge\omega,\quad\omega\in\ker\theta_w, \quad \xi \in \ker\theta_w.
\end{align*}
Then
\begin{align*}
0=\theta_w(\xi)=\theta_w(\frac{dS}{S} \wedge \omega) = \frac{\theta_w(dS)}{S}\wedge \omega -dS \wedge \frac{\theta_w(S)}{S} = \omega.
\end{align*}
Thus we conclude that $\xi=0$. Therefore,
\begin{align*}
\Omega_{A[S^{-1}]/\mathbb{C}}^\bullet=\ker\theta_w\oplus\frac{dS}{S}\wedge\ker\theta_w,
\end{align*}
Since the restriction of $\theta_w$ on the subspace of $\deg_w=0$ induces a cochain map, we get the desired decomposition.
\end{proof}

By Proposition \ref{Cayley-decomposition}, $\theta_w$ restricts to a cochain map
\begin{align*}
\xymatrix{\theta_w:\left(\Omega_{A[S^{-1}]/\mathbb{C}}^\bullet,d\right)_{0,(0)} \ar[r] & \left(\Omega_{A[S^{-1}]_{(0)}/\mathbb{C}}^\bullet[-1],d\right)_0}
\end{align*}
which admits a section
\begin{align*}
\xymatrix{\displaystyle\frac{dS}{S}\wedge-:\left(\Omega_{A[S^{-1}]_{(0)}/\mathbb{C}}^\bullet[-1],d\right)_0 \ar[r] & \left(\Omega_{A[S^{-1}]/\mathbb{C}}^\bullet,d\right)_{0,(0)}}
\end{align*}
On the other hand, by Lemma \ref{Cayley-theta}, there is a cochain map
\begin{align*}
\xymatrix{\theta_c:=\theta_{c,r-n}\circ\cdots\circ\theta_{c,1}:\left(\Omega_{A[S^{-1}]_{(0)}/\mathbb{C}}^\bullet,d\right)_0 \ar[r] & \left(\widehat{\Omega}_{A[S^{-1}]^\mathbf{G}/\mathbb{C}}^\bullet[n-r],d\right)}
\end{align*}

\begin{proposition}\label{tctw-surjective} The cochain map
\begin{align}\label{tctw}
\xymatrix{\theta_c\circ\theta_w:\left(\Omega_{A[S^{-1}]/\mathbb{C}}^\bullet,d\right)_{0,(0)} \ar[r] & \left(\widehat{\Omega}_{A[S^{-1}]^\mathbf{G}/\mathbb{C}}^\bullet,d\right)[n-r-1]}.
\end{align}
induces a surjection on the cohomology.
\end{proposition}
\begin{proof} For a toric variety $Y_\Phi$ coming from a fan $\Phi$, denote by $\CMcal{L}_{Y_\Phi/\mathbb{C}}^{\bullet,\bullet}$ its Ishida complex (cf. \cite[p.120]{TOCB}). Since $\Spec A[S^{-1}]\rightarrow\Spec A[S^{-1}]^\mathbf{G}$ is a toric morphism, it induces a commutative square
\begin{align*}
\xymatrix{
\left(\Omega_{A[S^{-1}]/\mathbb{C}}^\bullet\right)_{0,(0)} \ar[r] & \left(\CMcal{L}_{A[S^{-1}]/\mathbb{C}}^\bullet\right)_{0,(0)} \\
\widehat{\Omega}_{A[S^{-1}]^\mathbf{G}/\mathbb{C}}^\bullet \ar[r] \ar[u] & \CMcal{L}_{A[S^{-1}]^\mathbf{G}/\mathbb{C}}^\bullet \ar[u]
}
\end{align*}
where the rows are quasi-isomorphisms by \cite[Theorem 3.6]{TOCB}. Here, taking $(\deg_c=0,\deg_w=0)$-subcomplex does not change the cohomology because the relations in Lemma \ref{Cayley-theta} are still valid so the proof of Proposition \ref{dRconcent} works for $\CMcal{L}_{A[S^{-1}]/\mathbb{C}}^\bullet$.\\
\indent Since $\theta$'s commute with the de Rham differential,
\begin{align*}
\xymatrix{\theta_c\circ\theta_w:\left(\CMcal{L}_{A[S^{-1}]/\mathbb{C}}^{\bullet,q}\right)_{0,(0)} \ar[r] & \CMcal{L}_{A[S^{-1}]^\mathbf{G}/\mathbb{C}}^{\bullet,q}[n-r-1]}
\end{align*}
is well-defined. Moreover, since $\theta$'s commute with themselves, $\theta_c\circ\theta_w$ above gather to define
\begin{align*}
\xymatrix{
\left(\Omega_{A[S^{-1}]/\mathbb{C}}^\bullet\right)_{0,(0)} \ar[d]_-{\theta_c\circ\theta_w} \ar[r] & \left(\CMcal{L}_{A[S^{-1}]/\mathbb{C}}^\bullet\right)_{0,(0)} \ar[d]^-{\theta_c\circ\theta_w} \\
\widehat{\Omega}_{A[S^{-1}]^\mathbf{G}/\mathbb{C}}^\bullet[n-r-1] \ar[r] & \CMcal{L}_{A[S^{-1}]^\mathbf{G}/\mathbb{C}}^\bullet[n-r-1]
}
\end{align*}
Moreover, each component of the Ishida complex is merely the log de Rham complex of intersections of torus-invariant divisors, we have
\begin{align*}
\CMcal{L}_{A[S^{-1}]^\mathbf{G}/\mathbb{C}}^\bullet=\ker\left(\xymatrix{\theta_c\circ\theta_w:\left(\CMcal{L}_{A[S^{-1}]/\mathbb{C}}^\bullet\right)_{0,(0)} \ar[r] & \CMcal{L}_{A[S^{-1}]^\mathbf{G}/\mathbb{C}}^\bullet[n-r-1]}\right)
\end{align*}
by Proposition \ref{reductive-quotient-de Rham}. Now, enumerate $\Sigma(1)=\{\rho_1,\cdots,\rho_r\}$ in such a way that $\{\rho_1,\cdots,\rho_{r-n}\}$ corresponds to the set $\{D_{\rho_1},\cdots,D_{\rho_{r-n}}\}$ of divisors defining a $\mathbb{C}$-basis for $\mathrm{Cl}(\mathbf{P}_\Sigma)_\mathbb{C}:=\mathrm{Cl}(\mathbf{P}_\Sigma)\otimes_\mathbb{Z}\mathbb{C}$. Following the identification (\ref{ray-dual}), we get a commutative diagram of $\mathbb{C}$-vector spaces:
\begin{align*}
\xymatrix{
\mathrm{Hom}(\mathrm{Cl}(\mathbf{P}_\Sigma),\mathbb{C}) \ar[r] & \mathrm{Hom}(\mathbb{Z}^{\oplus\Sigma(1)},\mathbb{C}) & & \\
\mathrm{Cl}(\mathbf{P}_\Sigma)_\mathbb{C} \ar@{-->}[r] \ar[u]^-\wr & \displaystyle\bigoplus_{\rho\in\Sigma(1)}\mathbb{C} D_\rho \ar[u]^-\wr & & [D_{\rho_i}] \ar@{|->}[r] & D_{\rho_i}
}
\end{align*}
Also, denote $x_i:=x_{\rho_i}$ for $1\leq i\leq r$. Consequently,
\begin{align*}
\xymatrix{\displaystyle\frac{dS}{S}\wedge\frac{dx_1}{x_1}\wedge\cdots\wedge\frac{dx_{r-n}}{x_{r-n}}\wedge:\CMcal{L}_{A[S^{-1}]^\mathbf{G}/\mathbb{C}}^\bullet[n-r-1] \ar[r] & \left(\CMcal{L}_{A[S^{-1}]/\mathbb{C}}^\bullet\right)_{0,(0)}}
\end{align*}
give a section of $\theta_c\circ\theta_w$. This shows that $\theta_c\circ\theta_w$ gives a cohomology surjection.
\end{proof}

\begin{proposition}\label{tctw-rank1-isom} If $\mathrm{rank}(\mathrm{Cl}(\mathbf{P}_\Sigma))=1$, then
\begin{align*}
\xymatrix{\theta_c:\left(\Omega_{A[S^{-1}]_{(0)}/\mathbb{C}}^\bullet,d\right)_0 \ar[r] & \left(\widehat{\Omega}_{A[S^{-1}]^\mathbf{G}/\mathbb{C}}^\bullet[-1],d\right)}
\end{align*}
induces an isomorphism for every $i\in\mathbb{Z}$:
\begin{align*}
H^i\left(\Omega_{A[S^{-1}]_{(0)}/\mathbb{C}}^\bullet,d\right)\cong H^i\left(\left(\Omega_{A[S^{-1}]_{(0)}/\mathbb{C}}^\bullet,d\right)_0\right)\cong H_{\widehat{\mathrm{dR}}}^i(\mathbf{P}(\CMcal{E})\setminus X_S)\oplus H_{\widehat{\mathrm{dR}}}^{i-1}(\mathbf{P}(\CMcal{E})\setminus X_S)
\end{align*}
where $\theta_c$ is identified with the projection onto the second factor. Consequently, we get an isomorphism
\begin{align*}
H^i\left(\Omega_{A[S^{-1}]/\mathbb{C}}^\bullet,d\right)\cong H_{\widehat{\mathrm{dR}}}^i(\mathbf{P}(\CMcal{E})\setminus X_S)\oplus H_{\widehat{\mathrm{dR}}}^{i-1}(\mathbf{P}(\CMcal{E})\setminus X_S)^{\oplus2}\oplus H_{\widehat{\mathrm{dR}}}^{i-2}(\mathbf{P}(\CMcal{E})\setminus X_S)
\end{align*}
for every $i\in\mathbb{Z}$.
\end{proposition}
\begin{proof} If $\mathrm{Hom}(\mathrm{Cl}(\mathbf{P}_\Sigma),\mathbb{Z})\cong\mathbb{Z}$, then there is only one $\theta_c$ with the corresponding $\deg_c$. In this special case, $d_i:=\deg_cG_i>0$ for every $i=1,\cdots,k$. Denote
\begin{align*}
A_i:=A[(y_iG_i)^{-1}],\quad i=1,\cdots,k
\end{align*}
so that we have
\begin{align*}
A_i[S^{-1}]^\mathbf{G}=A\left[S^{-1},(y_iG_i)^{-1}\right]^\mathbf{G}\cong A[S^{-1}]^\mathbf{G}\left[\frac{S}{y_iG_i}\right]
\end{align*}
\begin{align*}
A_i[S^{-1}]_{(0)}=A\left[S^{-1},(y_iG_i)^{-1}\right]_{(0)}\cong A[S^{-1}]_{(0)}\left[\frac{S}{y_iG_i}\right].
\end{align*}
Moreover, $\theta_c$ over $\mathbf{P}(\CMcal{E})\setminus(X_S\cup X_{y_iG_i})\cong\Spec A_i[S^{-1}]^\mathbf{G}$ admits a section given by
\begin{align}\label{rankone-thetasection}
\xymatrix{\displaystyle\frac{1}{d_i}\frac{dG_i}{G_i}\wedge-:\left(\widehat{\Omega}_{A_i[S^{-1}]^\mathbf{G}/\mathbb{C}}^\bullet,d\right)[-1] \ar[r] & \left(\Omega_{A_i[S^{-1}]_{(0)}/\mathbb{C}}^\bullet,d\right)_0}
\end{align}
Now the open coverings
\begin{align*}
\left\{\xymatrix{\Spec A_i[S^{-1}]^\mathbf{G} \ar[r] & \Spec A[S^{-1}]^\mathbf{G}}\right\}_{i=1,\cdots,k}\quad\left\{\xymatrix{\Spec A_i[S^{-1}]_{(0)} \ar[r] & \Spec A[S^{-1}]_{(0)}}\right\}_{i=1,\cdots,k}
\end{align*}
give $0$-coskeletal cosimplicial algebras $A[S^{-1}]^\mathbf{G}\rightarrow A^\bullet[S^{-1}]^\mathbf{G}$ and $A[S^{-1}]_{(0)}\rightarrow A^\bullet[S^{-1}]_{(0)}$ respectively. Taking the K\"ahler differentials, we get a commutative diagram of simplicial de Rham complexes:
\begin{align*}
\xymatrix{
\left(\Omega_{A[S^{-1}]_{(0)}/\mathbb{C}}^\bullet,d\right) \ar[r] & \left(\Omega_{A^\bullet[S^{-1}]_{(0)}/\mathbb{C}}^\bullet,d\right) \\
\left(\Omega_{A[S^{-1}]_{(0)}/\mathbb{C}}^\bullet,d\right)_0 \ar[d]_-{\theta_c} \ar[r] \ar@{^(->}[u] & \left(\Omega_{A^\bullet[S^{-1}]_{(0)}/\mathbb{C}}^\bullet,d\right)_0 \ar[d]_-{\theta_c} \ar@{^(->}[u] \\
\left(\widehat{\Omega}_{A[S^{-1}]^\mathbf{G}/\mathbb{C}}^\bullet,d\right)[-1] \ar[r] & \left(\widehat{\Omega}_{A^\bullet[S^{-1}]^\mathbf{G}/\mathbb{C}}^\bullet,d\right)[-1]
}
\end{align*}
where the $\theta_c$ in the right column is given by the restriction of the left column. The inclusions are quasi-isomorphisms by Proposition \ref{dRconcent}. Since the cosimplicial algebras come from Zariksi open coverings, the first and the third rows are quasi-isomorphisms. Consequently, the second row is a quasi-isomorphism as well. Moreover, (\ref{rankone-thetasection}) gives a map
\begin{align*}
\xymatrix{\displaystyle\prod_{i=1}^k\frac{1}{d_i}\frac{dG_i}{G_i}\wedge-:\prod_{i=1}^k\left(\widehat{\Omega}_{A_i[S^{-1}]^\mathbf{G}/\mathbb{C}}^\bullet,d\right)[-1] \ar[r] & \displaystyle\prod_{i=1}^k\left(\Omega_{A_i[S^{-1}]_{(0)}/\mathbb{C}}^\bullet,d\right)_0}
\end{align*}
at cosimplicial degree $0$ which extends uniquely to the whole cosimplicial de Rham complexes to give a section of $\theta_c$ on the associated \v{C}ech-de Rham complexes by Proposition \ref{cosimplicial-resolution}. Since
\begin{align*}
\left(\widehat{\Omega}_{A[S^{-1}]^\mathbf{G}/\mathbb{C}}^\bullet,d\right)=\ker\left(\xymatrix{\theta_c:\left(\Omega_{A[S^{-1}]_{(0)}/\mathbb{C}}^\bullet,d\right)_0 \ar[r] & \left(\Omega_{A[S^{-1}]_{(0)}/\mathbb{C}}^\bullet[-1],d\right)_0}\right)
\end{align*}
by Proposition \ref{Cayley-decomposition} and Proposition \ref{reductive-quotient-de Rham}, we obtain the assertion.
\end{proof}

\subsection{Twisted de Rham complexes}\label{sec2.4}
In this section, we will describe $H_{\widehat{\mathrm{dR}}}^\bullet(\mathbf{P}(\CMcal{E})\setminus X_S)$ in terms of the twisted de Rham complex
\begin{align*}
\left(\Omega_{A/\mathbb{C}}^\bullet,D_S:=d+dS\wedge-\right)
\end{align*}
which is graded by
\begin{align*}
\deg dq_i:=\deg q_i,\quad i=1,\cdots,r+k.
\end{align*}
To make the argument clear, we first consider the case where $\Sigma$ is smooth. In this case, $\mathbf{P}_\Sigma$ and $\mathbf{P}(\CMcal{E})$ become smooth so from Corollary \ref{refldR} we get
\begin{align*}
H_{\widehat{\mathrm{dR}}}^\bullet(\mathbf{P}(\CMcal{E})\setminus X_S)\cong H_\mathrm{dR}^\bullet(\mathbf{P}(\CMcal{E})\setminus X_S).
\end{align*}
Moreover, the quotient map $U_\Sigma\rightarrow\mathbf{P}_\Sigma$ becomes a $\mathbf{G}_\Sigma$-torsor\footnote{If $\Sigma$ is not smooth, this fails.} (cf. \cite[p.293]{BatCox}). Hence the quotient map $\Spec A[S^{-1}]\rightarrow\mathbf{P}(\CMcal{E})\setminus X_S$ becomes a $\mathbf{G}$-torsor, i.e. a principal $\mathbf{G}$-bundle (recall that $\mathbf{G}=\mathbf{G}_\Sigma\times_\mathbb{C}\mathbf{G}_m$). Since $\mathbf{G}$ is an affine scheme, the argument of Corollary \ref{dR-affine-bundle} works to give an isomorphism:
\begin{align}\label{torsor-coh}
H_\mathrm{dR}^\bullet(\Spec A[S^{-1}])\cong H_\mathrm{dR}^\bullet(\mathbf{P}(\CMcal{E})\setminus X_S)\otimes_\mathbb{C} H_\mathrm{dR}^\bullet(\mathbf{G}).
\end{align}
which preserves the $\mathbf{G}$-action where $\mathbf{G}$ acts on $H_\mathrm{dR}^\bullet(\mathbf{G})$ by multiplication. Consequently, the quotient map $\Spec A[S^{-1}]\rightarrow\mathbf{P}(\CMcal{E})\setminus X_S$ induces an isomorphism
\begin{align*}
H_\mathrm{dR}^\bullet(\mathbf{P}(\CMcal{E})\setminus X_S)\cong H_\mathrm{dR}^\bullet(\Spec A[S^{-1}])^{\mathbf{G}}
\end{align*}
Next, we consider the following map.
\begin{definition}\label{rho-defn} Define the $\mathbb{C}$-linear map
\begin{align*}
\xymatrix{\rho_S:\left(\Omega_{A/\mathbb{C}}^\bullet,D_S\right)_{0,(+)} \ar[r] & \left(\Omega_{A[S^{-1}]/\mathbb{C}}^\bullet,d\right)_{0,(0)}}
\end{align*}
by the following formula on $\deg$-homogeneous elements:
\begin{align*}
\rho_S(\xi):=(-1)^{(\deg_w\xi-1)!}\frac{\xi}{S^{\deg_w\xi}}
\end{align*}
\end{definition}

\begin{remark} 
We cannot extend $\rho_S$ to $(\Omega_{A/\mathbb{C}}^\bullet,D_S)_0$.
In order to extend $\rho_S$ to the $\deg_c=0$ complex, we have to choose the value manually, since $(-1)!$ is not a well-defined number. Since $\Sigma$ is a simplicial fan, $1\in R_\Sigma$ is the only $(\deg_c,\deg_w)=(0,0)$ element up to scalar multiplication by $\mathbb{C}$. Hence it suffices to consider $\rho_S(1)$ only. For $\rho_S$ to be a cochain map, $\rho_S(1)$ must satisfy
\begin{align*}
d\rho_S(1)=\rho_S(D_S(1))=\rho_S(dS)=\frac{dS}{S}
\end{align*}
so $\rho_S(1)=\log S$. However, this is impossible in $A[S^{-1}]$.
\end{remark}

\begin{remark}\label{rho-restriction}
Note that $\deg_wS=1$, $\deg_cS=0$ and hence
\begin{align}
\left(\Omega_{A/\mathbb{C}}^{\bullet>0},D_S\right)_0\subseteq\left(\Omega_{A/\mathbb{C}}^\bullet,D_S\right)_{0,(+)}.
\end{align}
Consequently, $\rho_S$ in Definition \ref{rho-defn} restricts to
\begin{align*}
\xymatrix{\displaystyle\rho_S:\left(\Omega_{A/\mathbb{C}}^{\bullet>0},D_S\right)_0 \ar[r] & \displaystyle\left(\Omega_{A[S^{-1}]/\mathbb{C}}^{\bullet>0},d\right)_{0,(0)}.}
\end{align*}
\end{remark}

By Proposition \ref{dRconcent}, the inclusion
\begin{align*}
\xymatrix{\left(\Omega_{A[S^{-1}]/\mathbb{C}}^\bullet,d\right)_{0,(0)} \ar@{^(->}[r] & \left(\Omega_{A[S^{-1}]/\mathbb{C}}^\bullet,d\right)}
\end{align*}
is a quasi-isomorphism. On the other hand, taking the contraction with Euler vector fields (\ref{tctw}), we get
\begin{align*}
\xymatrix{\theta_c\circ\theta_w:\left(\Omega_{A[S^{-1}]/\mathbb{C}}^\bullet,d\right)_{0,(0)} \ar[r] & \left(\Omega_{A[S^{-1}]^\mathbf{G}/\mathbb{C}}^\bullet,d\right)[-r+n-1]}
\end{align*}
which on the cohomology induces the projection onto the identity component in (\ref{torsor-coh}):
\begin{align*}
\xymatrix{H_\mathrm{dR}^\bullet(\Spec A[S^{-1}]) \ar[r] & H_\mathrm{dR}^{\bullet-r+n-1}(\mathbf{P}(\CMcal{E})\setminus X_S)}
\end{align*}
where we use $H_\mathrm{dR}^{r-n+1}(\mathbf{G}_\Sigma\times_\mathbb{C}\mathbf{G}_m)\cong\mathbb{C}$. In particular, this induces an isomorphism
\begin{align*}
\xymatrix{H_\mathrm{dR}^{r+k}(\Spec A[S^{-1}]) \ar[r]^-\sim & H_\mathrm{dR}^{n+k-1}(\mathbf{P}(\CMcal{E})\setminus X_S)}.
\end{align*}

Combining all the maps constructed so far, we get the following diagram:
\begin{align*}
\xymatrix{
\left(\Omega_{A/\mathbb{C}}^\bullet,D_S\right)_{0,(0)} \ar@{^(->}[d] \ar[r]^-{\rho_S} & \left(\Omega_{A[S^{-1}]/\mathbb{C}}^\bullet,d\right)_{0,(0)} \ar[d]^-{\theta_c\circ\theta_w} \\
\left(\Omega_{A/\mathbb{C}}^\bullet,D_S\right) & \left(\Omega_{A[S^{-1}]^\mathbf{G}/\mathbb{C}}^\bullet,d\right)[-r+n-1]
}
\end{align*}
where the inclusion induces a surjection on cohomology with $1$-dimensional kernel generated by the class of $dS$ (Lemma \ref{TdRconcent}). From Proposition \ref{TdRCayley-rho2}, taking $H^{r+k}$, we get the following isomorphism:
\begin{align*}
\xymatrix{\theta_c\circ\theta_w\circ\rho_S:H^{r+k}\left(\Omega_{A/\mathbb{C}}^\bullet,D_S\right) \ar[r]^-\sim & H_\mathrm{dR}^{n+k-1}(\mathbf{P}(\CMcal{E})\setminus X_S)}.
\end{align*}

Now let $\Sigma$ be simplicial (not necessarily smooth). In this case, the $\mathbf{G}_\Sigma$-action on $U_\Sigma$ may admit finite abelian stabilizer at each point so the quotient map $U_\Sigma\rightarrow\mathbf{P}_\Sigma$ may not be a $\mathbf{G}_\Sigma$-torsor. Consequently, the isomorphism (\ref{torsor-coh}) may not hold. Nevertheless, the maps $\rho$ and $\theta$ are explicitly defined if we replace the algebraic de Rham complexes by their reflexive hulls, which is canonically isomorphic to the  original de Rham complex by Corollary \ref{refldR}. Along this line, we only ensure the surjectivity of $\theta_c\circ\theta_w$ (Proposition \ref{tctw-surjective}). However, we may prove that in some cases we still get the desired isomorphism.
\begin{theorem}\label{TdRCayley-comparison} For $i\geq2$, the map $\rho_S$ induces a surjection
\begin{align*}
\xymatrix{\theta_c\circ\theta_w\circ\rho_S:H^i\left(\Omega_{A/\mathbb{C}}^\bullet,D_S\right) \ar[r] & H_{\widehat{\mathrm{dR}}}^{i-r+n-1}(\mathbf{P}(\CMcal{E})\setminus X_S)}.
\end{align*}
Moreover, the $(r+k)$-th degree
\begin{align*}
\xymatrix{\theta_c\circ\theta_w\circ\rho_S:H^{r+k}\left(\Omega_{A/\mathbb{C}}^\bullet,D_S\right) \ar[r] & H_{\widehat{\mathrm{dR}}}^{n+k-1}(\mathbf{P}(\CMcal{E})\setminus X_S)}
\end{align*}
becomes an isomorphism in each of the following cases.
\begin{quote}
(1) The quotient map $\Spec A[S^{-1}]\rightarrow\mathbf{P}(\CMcal{E})\setminus X_S$ induces a $\mathbf{G}$-equivariant isomorphism
\begin{align}\label{fcase}
H_\mathrm{dR}^\bullet(\Spec A[S^{-1}])\cong H_{\widehat{\mathrm{dR}}}^\bullet(\mathbf{P}(\CMcal{E})\setminus X_S)\otimes_\mathbb{C} H_\mathrm{dR}^\bullet(\mathbf{G}).
\end{align}
This is the case if $\Sigma$ is smooth or $\mathrm{rank}(\mathrm{Cl}(\mathbf{P}_\Sigma))=1$.\\
(2) $X_G\subseteq\mathbf{P}_\Sigma$ is a quasi-smooth complete intersection of ample hypersurfaces.
\end{quote}
\end{theorem}
\begin{proof} The first assertion follows from Proposition \ref{tctw-surjective} and Proposition \ref{TdRCayley-rho2}.\\
\indent (1) This follows from Proposition \ref{tctw-rank1-isom} and Proposition \ref{TdRCayley-rho2}.\\
\indent (2) By Corollary \ref{Dwork-finite}, Proposition \ref{TdR-qsm}, and Proposition \ref{normal-dR-compare}, the $(r+k)$-th degree of $\theta_c\circ\theta_w\circ\rho_S$ becomes a $\mathbb{C}$-linear surjection between $\mathbb{C}$-vector spaces of same finite dimension. Hence it necessarily becomes a $\mathbb{C}$-linear isomorphism.
\end{proof}
\begin{remark}
    In Theorem \ref{TdRCayley-comparison}, we remark that (1) and (2) cover different cases. For (1), we do not impose any condition on $X_G$, while we require that the fan $\Sigma$ defining $\mathbf{P}_\Sigma$ is nice enough so that \eqref{fcase} holds. On the other hand, for (2), we require some regularity condition on $X_G$, while we do not impose any further condition on $\Sigma$.
\end{remark}

\subsection{The proof in the first case}\label{sec2.5}

We study the behavior of the map $\rho_S$ on cohomology (Proposition \ref{TdRCayley-rho2}) to complete the proof of the case (1) of Theorem \ref{TdRCayley-comparison}.

\begin{lemma}\label{TdRconcent} The inclusion
\begin{align*}
\xymatrix{\left(\Omega_{A/\mathbb{C}}^\bullet,D_S\right)_0 \ar@{^(->}[r] & \left(\Omega_{A/\mathbb{C}}^\bullet,D_S\right)}
\end{align*}
is a quasi-isomorphism. On the other hand, the inclusion
\begin{align*}
\xymatrix{\left(\Omega_{A/\mathbb{C}}^\bullet,D_S\right)_{0,(+)} \ar@{^(->}[r] & \left(\Omega_{A/\mathbb{C}}^\bullet,D_S\right)_0}
\end{align*}
induces a surjection on cohomology spaces with one-dimensional kernel spanned by the class of $dS$.
\end{lemma}
\begin{proof}
Since $\deg_cdS=0$, the differential $D_S:=d+dS\wedge-$ is compatible with $\deg_c$ so the subcomplex is well-defined. Moreover, by Lemma \ref{Cayley-theta} (3), each $\xi\in\Omega_{A/\mathbb{C}}^\bullet$ homogeneous with respect to $\deg_c$ satisfies the relation
\begin{align*}
(D_S\theta_c+\theta_cD_S)\xi=(\deg_c\xi)\xi
\end{align*}
so if $D_S\xi=0$, then $\xi$ is in the image of $D_S$ unless $\deg_c\xi=0$. Hence the cohomology of $(\Omega_{A/\mathbb{C}}^\bullet,D_S)$ is concentrated in $\deg_c=0$ subcomplex.\\
\indent Since $\Sigma$ is a simplicial fan, $1\in R_\Sigma$ is the only $(\deg_c,\deg_w)=(0,0)$ element up to scalar multiplication by $\mathbb{C}$. Since
\begin{align*}
D_S(1)=dS,\quad D_S(dS)=0
\end{align*}
$1\in\Omega_{A/\mathbb{C}}^\bullet$ does not contribute to the cohomology and kills the class $[dS]$. On the other hand, if $\deg_wf>0$, then the equation
\begin{align*}
D_S(f)=dS\iff df=(1-f)dS
\end{align*}
has no polynomial solutions. Hence $dS$ defines a nontrivial class in the cohomology of the subcomplex with $(\deg_c=0,\deg_w>0)$.
\end{proof}

\begin{lemma}\label{rho-properties} We list the properties of $\rho_S$.
\begin{quote}
(1) $\rho_S$ is a $\mathbb{C}$-linear cochain map.\\
(2) $\rho_S$ commutes with $\theta_c$ and $\theta_w$. Here $\theta_w$ is regarded as a map of graded vector spaces.\\
(3) If $\xi_1$ and $\xi_2$ are $\deg_w$-homogeneous of positive degree, then
\begin{align*}
\rho_S(\xi_1\wedge\xi_2)=-\frac{(\deg_w\xi_1+\deg_w\xi_2-1)!}{(\deg_w\xi_1-1)!(\deg_w\xi_2-1)!}\rho_S(\xi_1)\wedge\rho_S(\xi_2).
\end{align*}
In particular, if $\xi$ is $\deg_w$-homogeneous of positive degree, then
\begin{align*}
\rho_S(S^i\xi)=(-1)^i\frac{(\deg_w\xi+i-1)!}{(\deg_w\xi-1)!}\rho_S(\xi)
\end{align*}
for every integer $i>0$.
\end{quote}
\end{lemma}
\begin{proof}
The results follow from direct computation.
\end{proof}
\begin{definition} Define the map
\begin{align*}
\xymatrix{\epsilon_{w,S}:=D_S\theta_w+\theta_wD_S:\left(\Omega_{A/\mathbb{C}}^\bullet,D_S\right) \ar[r] & \left(\Omega_{A/\mathbb{C}}^\bullet,D_S\right) }
\end{align*}
By Lemma \ref{Cayley-theta}, if $\xi$ is homogeneous with respect to $\deg_w$, then
\begin{align*}
\epsilon_{w,S}(\xi)=(\deg_w\xi+S)\xi.
\end{align*}
\end{definition}
\begin{lemma}\label{epsilon-properties} We list the properties of $\epsilon_{w,S}$.
\begin{quote}
(1) $\epsilon_{w,S}$ is a $\mathbb{C}$-linear cochain map.\\
(2) $\epsilon_{w,S}$ is injective.\\
(3) $\epsilon_{w,S}$ restricts to the bidegree $(\deg_c=0,\deg_w>0)$-subcomplex:
\begin{align*}
\xymatrix{\epsilon_{w,S}:\left(\Omega_{A/\mathbb{C}}^\bullet,D_S\right)_{0,(+)} \ar[r] & \left(\Omega_{A/\mathbb{C}}^\bullet,D_S\right)_{0,(+)}}
\end{align*}
If $\xi\in\left(\Omega_{A/\mathbb{C}}^\bullet\right)_{0,(+)}$ is a $\deg_w$-homogeneous element, then
\begin{align*}
S^i\xi\equiv(-1)^i\frac{(\deg_w\xi+i-1)!}{(\deg_w\xi-1)!}\xi\mod\epsilon_{w,S}\left(\Omega_{A/\mathbb{C}}^\bullet\right)_{0,(+)}.
\end{align*}
\end{quote}
\end{lemma}
\begin{proof}
(1) $\epsilon_{w,S}$ is $\mathbb{C}$-linear by construction and is a cochain map because
\begin{align*}
D_S\epsilon_{w,S}=D_S\theta_wD_S=(\epsilon_{w,S}-\theta_wD_S)D_S=\epsilon_{w,S}D_S
\end{align*}
\indent (2) If $\xi$ is a nonzero $\deg_w$-homogeneous element, then
\begin{align*}
\epsilon_{w,S}(\xi)=(\deg_w\xi+S)\xi
\end{align*}
is nonzero by our choice of $S$.\\
\indent (3) Since $D_S$ and $\theta_w$ preserve $\deg_c$, the first assertion follows. For the rest part, we proceed by induction on the power $i$ of $S$. The case $i=1$ follows immediately from construction. If $i>1$, then
\begin{align*}
S^i\xi&=S\cdot S^{i-1}\xi\\
&\equiv-(\deg_w\xi+i-1)S^{i-1}\xi\mod\im\epsilon_{w,S}\\
&\equiv-(\deg_w\xi+i-1)\cdot(-1)^{i-1}\frac{(\deg_w\xi+i-2)!}{(\deg_w\xi-1)!}\xi\mod\im\epsilon_{w,S}\\
&\equiv(-1)^i\frac{(\deg_w+i-1)!}{(\deg_w\xi-1)!}\xi\mod\im\epsilon_{w,S}
\end{align*}
where the third line follows from the induction hypothesis.\\
\indent (4) follows from (3) because $D_S$ and $\theta_w$ preserve $\deg_c$.
\end{proof}
\begin{proposition}\label{TdRCayley-rho} There is an exact sequence of cochain complexes
\begin{align*}
\xymatrix{0 \ar[r] & \left(\Omega_{A/\mathbb{C}}^\bullet,D_S\right)_{0,(+)} \ar[r]^-{\epsilon_{w,S}} & \left(\Omega_{A/\mathbb{C}}^\bullet,D_S\right)_{0,(+)} \ar[r]^-{\rho_S} & \left(\Omega_{A[S^{-1}]/\mathbb{C}}^\bullet,d\right)_{0,(0)} \ar[r] & 0}.
\end{align*}
Consequently, there is an exact sequence
\begin{align*}
\xymatrix{0 \ar[r] & H^i\left(\left(\Omega_{A/\mathbb{C}}^\bullet,D_S\right)_{0,(+)}\right) \ar[r]^-{\rho_S} & H^i\left(\Omega_{A[S^{-1}]/\mathbb{C}}^\bullet,d\right) \ar[r]^-\delta & H^{i+1}\left(\left(\Omega_{A/\mathbb{C}}^\bullet,D_S\right)_{0,(+)}\right) \ar[r] & 0}
\end{align*}
for every $i\in\mathbb{Z}$, and in particular,
\begin{align*}
H^0\left(\left(\Omega_{A/\mathbb{C}}^\bullet,D_S\right)_{0,(+)}\right)=0,\quad H^1\left(\left(\Omega_{A/\mathbb{C}}^\bullet,D_S\right)_{0,(+)}\right)=\mathbb{C}\cdot[dS].
\end{align*}
\end{proposition}
\begin{proof}
For the first part, note that the target of $\rho_S$ admits a $\mathbb{C}$-basis of the form
\begin{align*}
\left\{\frac{x^uy^v}{S^{|v|}}dx_I\wedge\frac{dy_J}{S^{|J|}}\ \middle|\ \deg_c\left(x^uy^vdx_I\wedge dy_J\right)=0\right\}
\end{align*}
which is in the image of $\rho_S$, the surjectivity of $\rho_S$ follows. The injectivity of $\epsilon_{w,S}$ follows from Lemma \ref{epsilon-properties}. If $\xi\in\left(\Omega_{A/\mathbb{C}}^\bullet\right)_{0,(+)}$ is $\deg_w$-homogeneous, then
\begin{align*}
(\rho_S\circ\epsilon_{w,S})(\xi)=\rho_S\left((\deg_w\xi)\xi+S\xi\right)=0
\end{align*}
by Lemma \ref{rho-properties}. Hence $\rho_S$ induces
\begin{align*}
\xymatrix{\displaystyle\overline{\rho}_S:\left(\frac{\left(\Omega_{A/\mathbb{C}}^\bullet\right)_{0,(+)}}{\epsilon_{w,S}\left(\Omega_{A/\mathbb{C}}^\bullet\right)_{0,(+)}},D_S\right) \ar[r] & \left(\Omega_{A[S^{-1}]/\mathbb{C}}^\bullet,d\right)_{0,(0)}}.
\end{align*}
Define the map of graded $\mathbb{C}$-vector spaces
\begin{align*}
\xymatrix{\sigma:\left(\Omega_{A[S^{-1}]/\mathbb{C}}^\bullet\right)_{0,(0)} \ar[r] & \displaystyle\frac{\left(\Omega_{A/\mathbb{C}}^\bullet\right)_{0,(+)}}{\epsilon_{w,S}\left(\Omega_{A/\mathbb{C}}^\bullet\right)_{0,(+)}}}
\end{align*}
by the formula
\begin{align*}
\sigma\left(\frac{x^uy^v}{S^{|v|}}dx_I\wedge\frac{dy_J}{S^{|J|}}\right):=\frac{(-1)^{|v|+|J|-1}}{(|v|+|J|-1)!}x^uy^vdx_I\wedge dy_J
\end{align*}
for $|v|+|J|>0$ together with the $\mathbb{C}$-linearity. From
\begin{align*}
&\quad\sigma\left(\frac{S^e}{S^e}\frac{x^uy^v}{S^{|v|}}dx_I\wedge\frac{dy_J}{S^{|J|}}\right)\\
&=\frac{(-1)^{|v|+|J|+e-1}}{(|v|+|J|+e-1)!}S^ex^uy^vdx_I\wedge dy_J\\
&\equiv\frac{(-1)^{|v|+|J|-1}}{(|v|+|J|-1)!}x^uy^vdx_I\wedge dy_J\mod\epsilon_{w,S}\left(\Omega_{A/\mathbb{C}}^\bullet\right)_{0,(+)}\\
&=\sigma\left(\frac{x^uy^v}{S^{|v|}}dx_I\wedge\frac{dy_J}{S^{|J|}}\right)
\end{align*}
we see that $\sigma$ is well-defined. Note that this forces $\sigma(1)=S$. By construction, $\overline{\rho}_S\sigma$ is the identity. Since $\coker\epsilon_{w,S}$ is spanned over $\mathbb{C}$ by $S^ex^uy^vdx_I\wedge dy_J$ with $e\geq0$ and $|v|+|J|>0$, and
\begin{align*}
S^ex^uy^vdx_I\wedge dy_J\equiv(-1)^e\frac{(|v|+|J|+e-1)!}{(|v|+|J|-1)!}x^uy^vdx_I\wedge dy_J\bmod\epsilon_{w,S}\left(\Omega_{A/\mathbb{C}}^\bullet\right)_{0,(+)}
\end{align*}
is in the image of $\sigma$, we conclude that $\sigma$ is surjective. Hence $\overline{\rho}_S$ and $\sigma$ are mutually inverses. Therefore we achieve the desired exactness.

For the second part, take the cohomology long exact sequence. Since $\epsilon_{w,S}$ is homotopic to zero by definition of $\epsilon_{w,S}$ and Lemma \ref{epsilon-properties}, we get the desired exact sequences. In particular, the long exact sequence begins with
\begin{align*}
\xymatrix{
0 \ar[r] & H^0\left(\left(\Omega_{A/\mathbb{C}}^\bullet,D_S\right)_{0,(+)}\right) \ar[d]^-0 &  \\
& H^0\left(\left(\Omega_{A/\mathbb{C}}^\bullet,D_S\right)_{0,(+)}\right) \ar[r]^-{\rho_S} & H^0\left(\Omega_{A[S^{-1}]/\mathbb{C}}^\bullet,d\right) \ar[r]^-\delta &  H^1\left(\left(\Omega_{A/\mathbb{C}}^\bullet,D_S\right)_{0,(+)}\right) \ar[d]^-0 \\
& & & H^1\left(\left(\Omega_{A/\mathbb{C}}^\bullet,D_S\right)_{0,(+)}\right).
}
\end{align*}
Hence we get the desired vanishing and the $\delta$ becomes an isomorphism
\begin{align*}
\xymatrix{\delta:H^0\left(\Omega_{A[S^{-1}]/\mathbb{C}}^\bullet,d\right) \ar[r]^-\sim &  H^1\left(\left(\Omega_{A/\mathbb{C}}^\bullet,D_S\right)_{0,(+)}\right)}.
\end{align*}
Since $\Spec A[S^{-1}]=\Spec R_\Sigma[y,S^{-1}]$ is connected, the right hand side is $1$-dimensional with a basis $[dS]$ coming from \ref{TdRconcent}.
\end{proof}
\begin{corollary}\label{Dwork-finite} The cohomology groups of the twisted de Rham complex
\begin{align*}
H^i\left(\Omega_{A/\mathbb{C}}^\bullet,D_S\right)
\end{align*}
are finite-dimensional $\mathbb{C}$-vector spaces for every $i\in\mathbb{Z}$. In particular,
\begin{align*}
H^0\left(\Omega_{A/\mathbb{C}}^\bullet,D_S\right)=0,\quad H^1\left(\Omega_{A/\mathbb{C}}^\bullet,D_S\right)=0.
\end{align*}
\end{corollary}
\begin{proof}
By Lemma \ref{TdRconcent}, we may compute the cohomology of the twisted de Rham complex by using $\left(\Omega_{A/\mathbb{C}}^\bullet,D_S\right)_{0,(+)}$. Hence the results follow from Proposition \ref{TdRCayley-rho} and the finiteness of the algebraic de Rham cohomology of smooth $\mathbb{C}$-algebras \cite[Theorem 3.1]{MFdR}.
\end{proof}

Now we will show that there is an isomorphism (Proposition \ref{TdRCayley-rho2}):
$$
H^{i+1}\left(\Omega_{A/\mathbb{C}}^\bullet,D_S\right) \cong H^i\left(\Omega_{A[S^{-1}]_{(0)}/\mathbb{C}}^\bullet,d\right) 
$$
for any $i \in \Z$.

\begin{definition} Define the $\mathbb{C}$-linear map
\begin{align*}
\xymatrix{\chi:\left(\Omega_{A/\mathbb{C}}^\bullet\right)_{0,(+)} \ar[r] & \left(\Omega_{A/\mathbb{C}}^\bullet\right)_{0,(+)}}
\end{align*}
as follows: If $\xi$ is a $\deg_w\xi$-homogeneous of positive degree, then denote
\begin{align*}
\chi_\xi:=\left\{\begin{array}{ll}
0 & \textrm{if}\quad\deg_w\xi=1\\
\displaystyle-\left(1+\frac{1}{2}+\cdots+\frac{1}{\deg_w\xi-1}\right) & \textrm{if}\quad\deg_w\xi>1
\end{array}\right.
\end{align*}
and define $\chi(\xi):=\chi_\xi\cdot\xi$ on $\deg_w\xi$-homogeneous elements.
\end{definition}
\begin{lemma}\label{chi-properties} As a cochain map, we have
\begin{align*}
\rho_S\circ(\chi D_S-D_S\chi)=\frac{dS}{S}\wedge\rho_S.
\end{align*}
\end{lemma}
\begin{proof}
If $\xi$ is a $\deg_w$-homogeneous element, then
\begin{align*}
(\chi D_S-D_S\chi)(\xi)&=\chi_\xi d\xi+\chi_{dS\wedge\xi} dS\wedge\xi-\chi_\xi d\xi-\chi_\xi dS\wedge\xi=-\frac{1}{\deg_w\xi}dS\wedge\xi
\end{align*}
because $\deg_w(dS\wedge\xi)=\deg_w\xi+1$. By Lemma \ref{rho-properties}, this gives
\begin{align*}
(\rho_S\circ(\chi D_S-D_S\chi))(\xi)&=-\frac{1}{\deg_w\xi}\rho_S(dS\wedge\xi)\\
&=\frac{1}{\deg_w\xi}\frac{(\deg_w\xi)!}{(\deg_w\xi-1)!}\rho_S(dS)\wedge\rho_S(\xi)\\
&=\frac{dS}{S}\wedge\rho_S(\xi)
\end{align*}
so the lemma follows.
\end{proof}

\begin{lemma}\label{rho-homotopy} With Notation \ref{notation-deg}, the square
\begin{align*}
\xymatrixcolsep{4pc}\xymatrix{
\left(\Omega_{A/\mathbb{C}}^\bullet,D_S\right)_{0,(+)} \ar@{=}[d] \ar[r]^-{\frac{dS}{S}\wedge\theta_w\rho_S} & \displaystyle\frac{dS}{S}\wedge\left(\Omega_{A[S^{-1}]_{(0)}/\mathbb{C}}^\bullet,d\right)_0 \ar@{^(->}[d] \\
\left(\Omega_{A/\mathbb{C}}^\bullet,D_S\right)_{0,(+)} \ar[r]_-{\rho_S} & \left(\Omega_{A[S^{-1}]/\mathbb{C}}^\bullet,d\right)_{0,(0)}
}
\end{align*}
commutes up to homotopy where we use the identification coming from the decomposition in Proposition \ref{Cayley-decomposition}.
\end{lemma}
\begin{proof}
Using Lemma \ref{rho-properties}, Lemma \ref{epsilon-properties}, Proposition \ref{TdRCayley-rho}, and Lemma \ref{chi-properties}, we get
\begin{align*}
d\rho_S\theta_w\chi+\rho_S\theta_w\chi D_S&=\rho_SD_S\theta_w\chi+\rho_S\theta_w\chi D_S\\
&=\rho_S(D_S\theta_w+\theta_wD_S)\chi+\rho_S\theta(\chi D_S-D_S\chi)\\
&=\rho_S\epsilon_{w,S}\chi+\theta_w\rho_S(\chi D_S-D_S\chi)\\
&=\theta_w\left(\frac{dS}{S}\wedge\rho_S\right)\\
&=\rho_S-\frac{dS}{S}\wedge\theta_w\rho_S
\end{align*}
where $\theta_w\rho_S$ maps into $\Omega_{A[S^{-1}]_{(0)}/\mathbb{C}}^\bullet$ by Lemma \ref{Cayley-decomposition}.
\end{proof}
\begin{proposition}\label{TdRCayley-rho2} With the exact sequence as in Proposition \ref{TdRCayley-rho} for $i\in\mathbb{Z}$:
\begin{align*}
\xymatrix{0 \ar[r] & H^i\left(\left(\Omega_{A/\mathbb{C}}^\bullet,D_S\right)_{0,(+)}\right) \ar[r]^-{\rho_S} & H^i\left(\Omega_{A[S^{-1}]/\mathbb{C}}^\bullet,d\right) \ar[r]^-\delta & H^{i+1}\left(\left(\Omega_{A/\mathbb{C}}^\bullet,D_S\right)_{0,(+)}\right) \ar[r] & 0}
\end{align*}
$\rho_S$ and $\delta$ above induce isomorphisms
\begin{align*}
\xymatrix{\delta:H^i\left(\Omega_{A[S^{-1}]_{(0)}/\mathbb{C}}^\bullet,d\right) \ar[r] & H^{i+1}\left(\left(\Omega_{A/\mathbb{C}}^\bullet,D_S\right)_{0,(+)}\right)}
\end{align*}
\begin{align*}
\xymatrix{\rho_S:H^i\left(\left(\Omega_{A/\mathbb{C}}^\bullet,D_S\right)_{0,(+)}\right) \ar[r] & \displaystyle\frac{dS}{S}\wedge H^{i-1}\left(\Omega_{A[S^{-1}]_{(0)}/\mathbb{C}}^\bullet,d\right)}
\end{align*}
where we use the identification of Proposition \ref{Cayley-decomposition}.
\end{proposition}
\begin{proof}
Suppose that $\xi\in\Omega_{A}^{i+1}$ is a $D_S$-closed form. If we take
\begin{align*}
\omega:=\theta_w\rho_S\xi=\rho_S\theta_w\xi,
\end{align*}
then $\omega\in\Omega_{A[S^{-1}]_{(0)}/\mathbb{C}}^i$ and
\begin{align*}
d\omega=d\theta_w\rho_S\xi=\theta_w\rho_SD_S\xi=0.
\end{align*}
Since $\omega$ is in the image of $\rho_S$, it represents a class in $H^i(\Omega_{A[S^{-1}]_{(0)}/\mathbb{C}}^\bullet,d)$. Now $\theta_w\xi$ is a lift of $\omega$ along $\rho_S$ and, since $D_S\xi=0$, we have
\begin{align*}
D_S\theta_w\xi=(D_S\theta_w+\theta_wD_S)\xi=\epsilon_{w,S}(\xi).
\end{align*}
Therefore, by the construction of connecting map $\delta$,
\begin{align*}
\delta[\omega]=\left[\epsilon_{w,S}^{-1}D_S\theta_w\xi\right]=[\xi]
\end{align*}
so $\delta$ restricted to $H^i(\Omega_{A[S^{-1}]_{(0)}/\mathbb{C}}^\bullet,d)$ is surjective.

On the other hand, $\rho_S$ defines an injection into $H^{i-1}(\Omega_{A[S^{-1}]_{(0)}/\mathbb{C}}^\bullet,d)$ by Proposition \ref{TdRCayley-rho} and Lemma \ref{rho-homotopy}. If $\xi\in H^{i-1}(\Omega_{A[S^{-1}]_{(0)}/\mathbb{C}}^\bullet,d)$, then there is $\widetilde{\xi}\in H^i(\Omega_{A[S^{-1}]_{(0)}/\mathbb{C}}^\bullet,d)$ with
\begin{align*}
\delta\widetilde{\xi}=\delta\left(\frac{dS}{S}\wedge\xi\right)
\end{align*}
by the surjectivity of $\delta$ observed above. Hence
\begin{align*}
\widetilde{\xi}-\frac{dS}{S}\wedge\xi\in\rho_SH^i\left(\left(\Omega_{A/\mathbb{C}}^\bullet,D_S\right)_{0,(+)}\right)
\end{align*}
but this implies $\widetilde{\xi}=0$ by Lemma \ref{rho-homotopy}. Hence $\rho_S$ is surjective as well, i.e. it is an isomorphism. By the identification of Proposition \ref{Cayley-decomposition}, this implies that $\delta$ is an isomorphism as well. The last assertion follows from Proposition \ref{Cayley-decomposition} together with Lemma \ref{TdRconcent}.
\end{proof}

\subsection{The proof in the second case}\label{sec2.6}

We need the following proposition to complete the proof of the case (2) of Theorem \ref{TdRCayley-comparison}.

\begin{proposition}\label{TdR-qsm} 
Let $X_G\subseteq\mathbf{P}_\Sigma$ be a quasi-smooth complete intersection of ample hypersurfaces. Then there is a $\mathbb{C}$-linear isomorphism
\begin{align*}
\frac{ \left(\Omega_{A/\mathbb{C}}^{r+k}\right)_0}{(d+dS\wedge-) \left(\Omega_{A/\mathbb{C}}^{r+k-1}\right)_0} \cong  H^{n+k-1} (\mathbf{P}(\CMcal{E})\setminus X_S).
\end{align*}
\end{proposition}
\begin{proof}
For the proof, we start to review briefly the result of \cite[section 9 and section 10]{BatCox}. Denote $M_{\mathbf{P}(\CMcal{E})}:=M\oplus\mathbb{Z}^{\oplus k-1}$ and $N_{\mathbf{P}(\CMcal{E})}:=\mathrm{Hom}(M_{\mathbf{P}(\CMcal{E})},\mathbb{Z})$. Following \cite[pp.135--136]{Mav}, we may construct a rational simplicial complete fan $\Sigma_{\mathbf{P}(\CMcal{E})}$ for $\mathbf{P}(\CMcal{E})$ in $N_{\mathbf{P}(\CMcal{E}),\mathbb{R}}$ with coordinate ring $A$. The exact sequence corresponding to \eqref{ray-cl-exact} is
\begin{align*}
\xymatrix{0 \ar[r] & M_{\mathbf{P}(\CMcal{E})} \ar[r] & \displaystyle\bigoplus_{i=1}^{r+k}\mathbb{Z}D_i \ar[r] & \mathrm{Cl}(\mathbf{P}(\CMcal{E})) \ar[r] & 0}
\end{align*}
where $D_i\subseteq\mathbf{P}(\CMcal{E})$ is the divisor cut out by $q_i\in A$. Denote $e_i\in N_{\mathbf{P}(\CMcal{E})}$ the ray generator corresponding to $D_i$. Choose a $\mathbb{Z}$-basis $m_1, \ldots, m_{n+k-1}$ for $M_{\mathbf{P}(\CMcal{E})}$. For each subset $I =\{i_1, \ldots, i_{n+k-1}\} \subset \{1, \ldots, r+k\}$ consisting of $n+k-1$ elements, denote (\cite[Defintion 9.1]{BatCox})
\begin{align*}
\det (e_I)&:=\det (\langle m_j, e_{i_k} \rangle_{1 \leq j, k \leq n+k-1})\\
dq_I&:=dq_{i_1} \wedge \cdots \wedge dq_{i_{n+k-1}}\\
\widehat{q}_I&:=\prod_{i \not\in I} q_i
\end{align*}
Then $\det(e_I)dz_I$ does not depend on how the elements of $I$ are ordered.
We define (\cite[Defintion 9.3]{BatCox})
\begin{align*}
\Omega_0 =\sum_{\substack{I\subseteq\{1,\cdots,r+k\} \\ |I|=n+k-1}} \det(e_I) \hat z_{I} dz_{I} \in \Omega_{A/\mathbb{C}}^{n+k-1}
\end{align*}
According to \cite[Theorem 9.7]{BatCox}, we have
\begin{align*}
H^0\left(\mathbf{P}(\CMcal{E}), \Omega_{\mathbf{P}(\CMcal{E})}^{n+k-1}(mX_S)\right)=\left\{ \frac{F\Omega_0 }{S^{m}}\ \middle|\ \deg F =m\deg S- \deg( \Omega_0 )\right\}, \quad m \geq 0.
\end{align*}

Given $1\leq i\leq r+k$, denote (\cite[Definition 9.8]{BatCox})
\begin{align*}
\Omega_i = \sum_{\substack{|J|=n+k-2 \\ J\cup\{i\}=I}} \det(e_{J \cup\{i\}})\hat z_{J \cup \{i\}} dz_J
\end{align*}
where $\det(e_{J \cup\{i\}})$ is computed by ordering the elements of $J\cup \{i\}$ so that $i$ is first so that $\det(e_{J \cup\{i\}})dz_J$ is well-defined.
Because $X_S$ is an ample hypersurface in $\mathbf{P}(\CMcal{E})$, \cite[Corollary 9.16]{BatCox} implies that
\begin{align*}
H^0\left(\mathbf{P}(\CMcal{E}), \Omega_{\mathbf{P}(\CMcal{E})}^{n+k-2}(mX_S)\right)=\left\{ \sum_{i=1}^{r+k}\frac{F_i \Omega_i }{S^{m}}\ \middle|\ \deg F_i =m\deg S - \deg( \Omega_i )\right\}, \quad m \geq 0.
\end{align*}

Let $F^\bullet$ be the Hodge filtration on $H^{n+k-1}(\mathbf{P}(\CMcal{E})\setminus X_S)$ and denote the associated graded piece by $\mathrm{Gr}_F^m H^{n+k-1}(\mathbf{P}(\CMcal{E})\setminus X_S)$.
\cite[Proposition 10.1]{BatCox} implies that
\begin{align}\label{coh}
\frac{H^0\left(\mathbf{P}(\CMcal{E}), \Omega_{\mathbf{P}(\CMcal{E})}^{n+k-1}((m+1)X_S)\right)}{H^0\left(\mathbf{P}(\CMcal{E}), \Omega_{\mathbf{P}(\CMcal{E})}^{n+k-1}(mX_S)\right)+ dH^0\left(\mathbf{P}(\CMcal{E}), \Omega_{\mathbf{P}(\CMcal{E})}^{n+k-2}(mX_S)\right)} =0, \quad m \geq n+k.
\end{align}
\cite[Corollary 10.2]{BatCox} says that
\begin{align}\label{cohtwo}
\mathrm{Gr}_F^{-m+n+k-1}H^{n+k-1}(\mathbf{P}(\CMcal{E})\setminus X_S)\cong\frac{H^0\left(\mathbf{P}(\CMcal{E}), \Omega_{\mathbf{P}(\CMcal{E})}^{n+k-1}((m+1)X_S)\right)}{H^0\left(\mathbf{P}(\CMcal{E}), \Omega_{\mathbf{P}(\CMcal{E})}^{n+k-1}(mX_S)\right)+dH^0\left(\mathbf{P}(\CMcal{E}),\Omega_{\mathbf{P}(\CMcal{E})}^{n+k-2}(mX_S)\right)}
\end{align}
for $ 0 \le m \le n+k-1$. Then we have $\mathbb{C}$-vector space isomorphisms
\begin{align}\label{iso}
\begin{aligned}
&\quad H^{n+k-1}(\mathbf{P}(\CMcal{E})\setminus X_S) \\
&\cong
\bigoplus_{0 \leq m\leq n+k-1} \mathrm{Gr}_F^m H^{n+k-1}(\mathbf{P}(\CMcal{E})\setminus X_S)\\
&\cong
\bigoplus_{ m \ge 0}\frac{H^0\left(\mathbf{P}(\CMcal{E}), \Omega_{\mathbf{P}(\CMcal{E})}^{n+k-1}((m+1)X_S)\right)}{H^0\left(\mathbf{P}(\CMcal{E}), \Omega_{\mathbf{P}(\CMcal{E})}^{n+k-1}(mX_S)\right)+dH^0\left(\mathbf{P}(\CMcal{E}), \Omega_{\mathbf{P}(\CMcal{E})}^{n+k-2}(mX_S)\right)} 
\text{ by \eqref{coh} and \eqref{cohtwo}}\\
&\cong
\frac{\displaystyle\sum_{m\geq 0}H^0\left(\mathbf{P}(\CMcal{E}), \Omega_{\mathbf{P}(\CMcal{E})}^{n+k-1}((m+1)X_S)\right)}
{\displaystyle\sum_{m\geq 0}d H^0\left(\mathbf{P}(\CMcal{E}), \Omega_{\mathbf{P}(\CMcal{E})}^{n+k-2}(mX_S)\right)}
\end{aligned}
\end{align}
where in the finial step we use the fact $H^0(\mathbf{P}(\CMcal{E}), \Omega_{\mathbf{P}(\CMcal{E})}^{n+k-1}(0X_S) )=0$, which follows from that statement that there is no polynomial $F \in A$ such that
\begin{align*}
\deg_w F = -\deg_w(\Omega_0)=-\sum_{i=1}^{r+k} \deg q_i < 0
\end{align*}
because $\deg_w x_i =0$ and $\deg_w y_i =1$. We now observe that there is an isomorphism of $\mathbb{C}$-vector spaces
\begin{align*}
\xymatrix{
\displaystyle\left(\Omega^{r+k}_{A[S^{-1}]/\mathbb{C}}\right)_{0,(0)} \ar[d]^-\wr_-{\alpha_{r+k}} & \displaystyle\frac{F}{S^m}dq_1 \wedge\cdots\wedge dq_{r+k} \ar@{|->}[d] \\
\displaystyle\sum_{m\geq 0}H^0\left(\mathbf{P}(\CMcal{E}), \Omega_{\mathbf{P}(\CMcal{E})}^{n+k-1}((m+1)X_S)\right) & \displaystyle\frac{F}{S^m} \Omega_0
}
\end{align*}
and an isomorphism of $\mathbb{C}$-vector spaces
\begin{align*}
\xymatrix{
\displaystyle\left(\Omega^{r+k-1}_{A[S^{-1}]/\mathbb{C}}\right)_{0,(0)} \ar[d]^-\wr_-{\alpha_{r+k-1}} & \displaystyle\sum_{j=1}^{r+k}\frac{F_j}{S^m}dq_1 \wedge\cdots\wedge\widehat{dq_j} \wedge \cdots \wedge dq_{r+k} \ar@{|->}[d] \\
\displaystyle\sum_{m\geq 0} H^0\left(\mathbf{P}(\CMcal{E}), \Omega_{\mathbf{P}(\CMcal{E})}^{n+k-2}(mX_S)\right) & \displaystyle\sum_{j=1}^{r+k}(-1)^{j-1}\frac{F_j}{S^m}\Omega_j.
}
\end{align*}
Moreover, they satisfy the following cochain condition.
\begin{lemma}\label{lem1}
We have
\be
d\circ\alpha_{r+k-1}=\alpha_{r+k}\circ d.
\ee
\end{lemma}
\begin{proof}
We compute
\begin{align*}
&\quad(d\circ\alpha_{r+k-1})\left(\sum_{j=1}^{r+k}\frac{F_j}{S^m}dq_1 \wedge \cdots \wedge \widehat{dq_j} \wedge \cdots \wedge dq_{r+k}\right)\\
&=d\left(\sum_{j=1}^{r+k}(-1)^{j-1}\frac{ F_j \Omega_j}{S^m}\right)\\
& = \sum_{j=1}^{r+k}(-1)^{j-1}\left(S\frac{\partial F_j}{\partial q_j}-mF_j\frac{\partial S }{\partial q_j}\right)\Omega_0\\
&= \alpha_{r+k} \left(\sum_{j=1}^{r+k}(-1)^{j-1}\frac{1}{S^{m+1}}\left(S\frac{\partial F_j}{\partial q_j}-mF_j\frac{\partial S }{\partial q_j}\right)dq_1\wedge \cdots \wedge dg_{r+k}\right) \\
&= \alpha_{r+k} \circ d \left( \sum_{j=1}^{r+k}\frac{F_j}{S^m}dq_1 \wedge \cdots \wedge \widehat{dq_j} \wedge \cdots \wedge dq_{r+k}\right)
\end{align*}
where the second equality follows from \cite[Lemma 10.7]{BatCox}.
\end{proof}

By Remark \ref{rho-restriction}, Lemma \ref{rho-properties}, and Lemma \ref{lem1}, we have the commutative diagram where the horizontal maps $\alpha_{r+k}$ and $\alpha_{r+k-1}$ are $\mathbb{C}$-linear isomorphisms:
\begin{align*}
\xymatrixcolsep{3pc}\xymatrix{
\left(\Omega_{A/\mathbb{C}}^{r+k}\right)_0 \ar[r]^-{\rho_{S,r+k}} & \left(\Omega^{r+k}_{A[S^{-1}]/\mathbb{C}}\right)_{0,(0)} \ar[r]^-{\alpha_{r+k}} & \displaystyle\bigoplus_{m\geq 0}H^0\left(\BP(\CMcal{E}), \Omega_{\BP(\CMcal{E})}^{n+k-1}((m+1)X_S)\right) \\
\left(\Omega_{A/\mathbb{C}}^{r+k-1}\right)_0 \ar[r]_-{\rho_{S,r+k-1}} \ar[u]^-{d+dS\wedge-} & \left( \Omega^{r+k-1}_{A[S^{-1}]/\mathbb{C}}\right)_{0,(0)} \ar[r]_-{\alpha_{r+k-1}} \ar[u]^-d & \displaystyle\bigoplus_{m\geq 0} H^0\left(\BP(\CMcal{E}), \Omega_{\BP(\CMcal{E})}^{n+k-2}(mX_S)\right). \ar[u]^-d
}
\end{align*}
By using Proposition \ref{TdRCayley-rho} (which describes the kernels of $\rho_{S,r+k}$ and $\rho_{S,r+k-1}$), we conclude that this commutative diagram and \eqref{iso} imply the desired isomorphism.
\end{proof}

\section{Weak primitive forms and formal flat $F$-manifolds}\label{sec3}


\subsection{Twisted de Rham complexes and dGBV algebras}\label{sec3.1}

We continue with the notations so far. We define the following $\mathbb{Z}$-graded super-commutative algebra $\mathcal{A}$ and differentials $\Delta$, $Q_S$ and $K_S$ as follows:
\begin{align}\label{dgbv}
\begin{aligned}
\mathcal{A}&:= A[\eta]={\mathbb{C}}[q_{1},q_2,\dots,q_{r+k}][\eta_{1},\eta_2,\dots,\eta_{r+k}],\\
\Delta&:=\xymatrix{\displaystyle\sum_{i=1}^{r+k}\frac{\partial}{\partial q_i}\frac{\partial}{\partial\eta_i}:\mathcal{A} \ar[r] & \mathcal{A}}\\
Q_S&:=\xymatrix{\displaystyle\sum_{i=1}^{r+k} \frac{\partial S}{\partial q_i}\frac{\partial}{\partial\eta_i}: \mathcal{A} \ar[r] & \mathcal{A}},\\
K_S&:=\xymatrix{Q_S+\Delta:\mathcal{A} \ar[r] & \mathcal{A}}.
\end{aligned}
\end{align}
where the $\eta_i$'s are other variables (of cohomological degree $-1$) corresponding to $q_i$'s. The additive cohomological $\mathbb{Z}$-grading of $\mathcal{A}$ is given by the rules ($|f|=m$ means $f\in\mathcal{A}^m$)
\begin{align*}
|q_i|=0, \ |\eta_i|=-1, \quad i=1, \dots, {r+k}.
\end{align*}
Note that $\cA^0=A$.
Then we have the following cochain complex
\begin{align*}
\xymatrix{0 \ar[r] & \mathcal{A}^{-(r+k)} \ar[r]^-{K_S} & \mathcal{A}^{-(r+k)+1} \ar[r]^-{K_S} & \cdots \ar[r]^-{K_S} & \mathcal{A}^0 \ar[r] & 0}
\end{align*}
We define the $\ell_2$-descendant of $K_S$ with respect to the product $\cdot$ as follows:
\begin{align*}
\ell_2^{K_S}(a,b):= K_S(ab)-K_S(a)b -(-1)^{|a|} a K_S(b), \quad a,b \in \mathcal{A}.
\end{align*}
Then $(\cA, \cdot, Q_S, K_S, \ell_2^{K_S})$ becomes a dGBV (differential Gerstenhaber-Batalin-Vilkovisky) algebra; see \cite[Definition 2.1]{KKP} for the definition of dGBV algebras.
Note that $\eta_\m \eta_\n= -\eta_\n \eta_\m$, and
\begin{align*}
|q_\m|+|\eta_\m|=-1,\quad\ch(q_\m)+\ch(\eta_\m)=0,\quad\wt(q_\m)+\wt(\eta_\m)=1
\end{align*}
with $\ch(q_\mu) = \deg_c (q_\mu)$ and $\wt(q_\mu) = \deg_w(q_{\mu})$. Following Notation \ref{notation-deg}, we write
\begin{align*}
\mathcal{A} = \bigoplus_{-(r+k)\leq j \leq 0}\bigoplus_{\lambda\in\mathrm{Cl}(\mathbf{P}_\Sigma)}\bigoplus_{w \geq 0}\mathcal{A}^j_{\l, (w)},
\end{align*}
i.e., $f\in \mathcal{A}^j_{\lambda,(w)}$ has $|f|=j$, $\ch(f)=\lambda$ and $\wt(f)=w$. Define the background charge to be
\begin{align*}
c_B := -\sum_{i=1}^{r+k} \ch(q_i)= -\sum_{i=1}^{r+k} \deg_c(q_i) \in \mathrm{Cl}(\mathbf{P}_\Sigma).
\end{align*}

\begin{definition} We define the following $\mathbb{C}$-linear map with $q^u= q_1^{u_1} \cdots q_{r+k}^{u_{r+k}}$ and $i_1<\cdots<i_\ell$:
\begin{align}\label{bvconversion}
\begin{aligned}
\xymatrix{
(\mathcal{A},Q_S+\Delta)[r+k] \ar[d]_-\mu & q^u\eta_{i_1}\cdots\eta_{i_\ell} \ar@{|->}[d] \\
\left(\Omega_{A/\mathbb{C}}^\bullet,d+dS\wedge-\right) & (-1)^{i_1+\cdots+i_\ell+\ell}q^udq_1\wedge\cdots\wedge\widehat{dq_{i_1}}\cdots\wedge\widehat{dq_{i_\ell}}\cdots\wedge dq_{r+k}
.}
\end{aligned}
\end{align}
\end{definition}

The usefulness
of the map $\mu$ lies in the fact that it changes the multiplication structure so that $(\cA, \cdot, K_S, \ell_2^{K_S})$ is a dGBV algebra but $(\Omega_{A/\mathbb{C}}^\bullet, \wedge, d+ dS\wedge-)$ does not induce a dGBV algebra structure.

\begin{lemma}\label{mu-properties} We list the properties of $\mu$:
\begin{quote}
(1) $\mu$ is a $\mathbb{C}$-linear cochain map.\\
(2) For each $1\leq i\leq r+k$ with $\displaystyle\partial_{\eta_i}:=\frac{\partial}{\partial\eta_i}$,
\begin{align*}
\xymatrixcolsep{3pc}\xymatrix{
\mathcal{A} \ar[d]_-\mu \ar[r]^-{\partial_{\eta_i}} & \mathcal{A} \ar[d]^-\mu \\
\Omega_{A/\mathbb{C}}^\bullet \ar[r]_-{dq_i\wedge-} & \Omega_{A/\mathbb{C}}^\bullet
}
\end{align*}
is commutative. Hence $\Delta$ corresponds to $d$ under $\mu$.\\
(3) For each $1\leq i\leq r+k$ with $\displaystyle\partial_{q_i}:=\frac{\partial}{\partial q_i}$,
\begin{align*}
\xymatrixcolsep{3pc}\xymatrix{
\mathcal{A} \ar[d]_-\mu \ar[r]^-{\eta_i} & \mathcal{A} \ar[d]^-\mu \\
\Omega_{A/\mathbb{C}}^\bullet \ar[r]_-{\langle\partial_{q_i},-\rangle} & \Omega_{A/\mathbb{C}}^\bullet
}
\end{align*}
is commutative. Hence $Q_S$ corresponds to $dS\wedge-$ under $\mu$.\\
(4) The ($\ch=\lambda$)-component corresponds to ($\deg_c=\lambda-c_B$)-component under $\mu$.
\end{quote}
\end{lemma}
\begin{proof} The results follow from direct computation.
\end{proof}

By Lemma \ref{mu-properties}, $\mu$ restricts to
\begin{align}
\xymatrix{\displaystyle\mu:\left(\mathcal{A}^{\bullet>-(r+k)}_{c_B},K_S\right) \ar[r] & \left(\Omega_{A/\mathbb{C}}^{\bullet>0},D_S\right)_0[r+k]}
\end{align}
Hence, by Remark \ref{rho-restriction}, we have the following sequence of maps:
\begin{align*}
\xymatrix{
\left(\mathcal{A}^{\bullet>-(r+k)}_{c_B},K_S\right) \ar[r]^-\mu & \left(\Omega_{A/\mathbb{C}}^{\bullet>0},D_S\right)_0[r+k] \ar[d]^-{\rho_S} & \\
& \left(\Omega_{A[S^{-1}]/\mathbb{C}}^{\bullet>0},d\right)_{0,(0)}[r+k] \ar[r]^-{\theta_c\circ\theta_w} & \left(\widehat{\Omega}_{A[S^{-1}]^\mathbf{G}/\mathbb{C}}^{\bullet>0},d\right)[n+k-1]}.
\end{align*}
In particular, if any of the assumptions in Theorem \ref{TdRCayley-comparison} holds, then we have isomorphisms:
\begin{align*}
\xymatrixcolsep{3pc}\xymatrix{\displaystyle\frac{A_{c_B}}{K_S(\mathcal{A}_{c_B}^{-1})}=H^0(\mathcal{A}_{c_B}^\bullet,K_S) \ar[r]^-\sim_-\mu & H^{r+k}\left(\Omega_{A/\mathbb{C}}^\bullet,D_S\right) \ar[r]^-\sim_-{\theta_c\circ\theta_w\circ\rho_S} & H^{n+k-1}(\mathbf{P}(\CMcal{E})\setminus X_S).}
\end{align*}

\subsection{The Calabi-Yau condition}\label{sec3.2}

Suppose from now on that $X_G\subseteq\mathbf{P}_\Sigma$ is a quasi-smooth complete intersection of ample hypersurfaces.

The cohomology of $\BP(\CMcal{E}) \setminus X_S$ is concentrated in charge $c_B$ part as we saw in Proposition \ref{TdR-qsm}.
%
In fact, according to \cite[Theorem 10.6]{BatCox}, we have
\begin{align*}
H^{n+k-1}(\BP(\CMcal{E})\setminus X_S) \cong \frac{A_{c_B}}{\mathrm{Jac}(S) \cap A_{c_B}}=A_{c_B}/Q_S(\mathcal{A}_{c_B}^{-1})
\end{align*}
where $\mathrm{Jac}(S)$ is the Jacobian ideal of $S$.
On the other hand, Proposition \ref{TdR-qsm} and Lemma \ref{mu-properties} say that
\begin{align*}
\xymatrix{\displaystyle H^{n+k-1} (\BP(\CMcal{E})\setminus X_S) \cong \frac{ \left(\Omega_{A/\mathbb{C}}^{r+k}\right)_{\deg_c=0} }{(d+dS\wedge) \left(\Omega_{A/\mathbb{C}}^{r+k-1}\right)_{\deg_c=0}}  \ar[r]^-{\mu^{-1}} & \displaystyle A_{c_B}/(Q_S+\Delta)(\mathcal{A}_{c_B}^{-1})}
\end{align*}
where $\mu$ was given in \eqref{bvconversion}.
This implies the following proposition.
\begin{proposition} We have a $\mathbb{C}$-linear isomorphism
\begin{align*}
\xymatrix{\displaystyle\frac{A_{c_B}}{Q_S(\mathcal{A}^{-1}_{c_B})} \ar[r]^-\sim & \displaystyle\frac{A_{c_B}}{K_S(\mathcal{A}^{-1}_{c_B})}.}
\end{align*}
\end{proposition}

We can relate this vector space to the primitive\footnote{$H^{n-k}_{\pr}(X_{G})$ is defined to the cokernel of the natural map $H^{n-k}(\BP_\Sigma) \to H^{n-k}(X_{G})$.} middle-dimensional cohomology $H^{n-k}_{\pr}(X_{G})$ of $X_{G}$: \cite[Proposition 3.2]{Mav} says that there is an exact sequence of mixed Hodge structures
\bea \label{exactprim}
0 \to H^{n-k-1}(\BP_\Sigma) \xrightarrow{\cup [X_{G}]} H^{n+k-1}(\BP_\Sigma) \to
H^{n+k-1}(\BP_\Sigma\setminus X_{G}) \to H_{\pr}^{n-k}(X_{G}) \to 0
\eea
where $\cup [X_{G}]$ is the cup product with the cohomology class $[X_S] \in H^{2k}(\BP(\CMcal{E}))$ of $X_S$.
Thus, if $\cup [X_{G}]$ is surjective, then we have
\be
H_{\pr}^{n-k}(X_{G}) \cong 
H^{n+k-1}(\BP_\Sigma\setminus X_{G}).
\ee
Therefore, if $\cup [X_{G}]$ is surjective, by \eqref{comp} we conclude 
\be
H_{\pr}^{n-k}(X_{G})  \cong H^{n+k-1}(\BP_\Sigma\setminus X_{G}) \cong_{\eqref{comp}}
H^{n+k-1}(\BP(\CMcal{E})\setminus X_S)\cong\frac{A_{c_B}}{\mathrm{Jac}(S) \cap A_{c_B}}=A_{c_B}/Q_S(\mathcal{A}_{c_B}^{-1}).
\ee

\begin{definition}
If $c_B=0$, then we call $X_{G} \subset \BP_\Sigma$ a Calabi-Yau toric complete intersection.
\end{definition}
If $X_{G}$ is Calabi-Yau, then $H^{n+k-1}(\BP_\Sigma \setminus X_{G})$ has a natural structure of commutative ring induced from $A_0/Q_S(\mathcal{A}_{0}^{-1})=A_{\deg_c=0}/Q_S(\mathcal{A}_{\deg_c=0}^{-1})$. We like to extend this commutative ring structure to (formal) vector fields on $H^{n+k-1}(\BP_\Sigma \setminus X_{G})$ so that it is compatible with the Lie bracket of vector fields in a certain way (so called, the (formal) flat $F$-manifold structure).

\subsection{Flat $F$-manifolds and Frobenius manifolds}\label{sec3.3}
Let $M$ be a complex manifold with the holomorphic structure sheaf $\CMcal{O}_M$.
Let $\cT_M$ be the holomorphic tangent sheaf.
A $(k,l)$-tensor means an $\CMcal{O}_M$-linear map $T: \cT_M^{\otimes k} \to \cT_M^{\otimes l}$.
The Lie derivative $\Lie_X$ along a vector field $X$ is a derivation on the sheaf of $(k,l)$-tensors, as well as the covariant derivative $\nabla_X T$ with respect to a connection $\nabla$ on $M$.
%
Thus the Lie derivative $\Lie_X T$ of a $(k,l)$-tensor along a vector field $X$ is again a $(k,l)$-tensor, as well as the covariant derivative $\nabla_X T$ with respect to a connection $\nabla$ on $M$.
Then $\nabla T$ can be viewed as a $(k+1,l)$-tensor.

\begin{definition}[flat $F$-manifolds]
A flat $F$-manifold is a quadruple $(M,\circ, e, \nabla)$ where $M$ is a complex connected manifold, $\circ$ is a commutative and associative $\CMcal{O}_M$-bilinear multiplication $\cT_M \times \cT_M \to \cT_M$, $e$ is a global unit vector field with respect to $\circ$, $\nabla$ is a flat connection on $\cT_M$, and $\circ$ is compatible with $\nabla$, i.e. each element of the pencil $(\nabla^z)_{z\in \C}$, defined by $\nabla_X^z(Y) = \nabla_X Y + z X\circ Y$ is flat and torsion-free:
\be
&& \nabla^z_{X}Y -\nabla^z_{Y}X=[X, Y] \quad (\text{torsion-free}),\\
&&\nabla^z_{X} \nabla^z_{Y} -\nabla^z_{Y} \nabla^z_{X} =\nabla^z_{[X, Y]} \quad (\text{flat}).
\ee
\end{definition}


\begin{definition}[Frobenius manifolds]
A Frobenius manifold is a tuple $(M, \circ, e, E, g)$ where where $M$ is a complex connected manifold with metric $g$, $\circ$ is a commutative and associative $\CMcal{O}_M$-bilinear multiplication $\cT_M \times \cT_M \to \cT_M$, $e$ is a global unit vector field with respect to $\circ$, and $E$ is another global vector field subject to the following conditions:
\begin{enumerate}[(1)]
\item (invariance) $g(X\circ Y, Z) = g(X, Y\circ Z)$
\item (potentiality) the (3,1)-tensor $\nabla^g \circ$ is symmetric where $\nabla^g$ is the Levi-Civita connection of $g$,
\item the metric $g$ is flat, i.e. $\nabla^g$ is a flat connection, $[\nabla^g_X,\nabla^g_Y]=\nabla^g_{[X,Y]}$,
\item $\nabla^g e=0$,
\item $\Lie_E(\circ) = \circ$ and $\Lie_E(g) = D \cdot g$ for some $D \in \C$.
\end{enumerate}
\end{definition}
%
%
%

One can similarly define the formal version of flat $F$-manifolds and Frobenius manifolds: we can consider the formal structure sheaf and the formal tangent bundle instead of the holomorphic structure sheaf and the holomorphic tangent bundle.
The following proposition is well-known (\cite{Her02}, \cite{HM}).
\begin{proposition}
If $(M, \circ, e, E,g)$ is a Frobenius manifold, then $(M, \circ, e, \nabla^g)$ is a flat $F$-manifold.
\end{proposition}

Let $M$ be an affine space of finite dimension over $\mathbb{C}$, whose formal structure sheaf and formal tangent sheaf are denoted by ${O}_M$ and ${T}_{M}$ respectively. 
%
Let $t_M := \{t^\alpha\}$ be the formal coordinates on $M$. Then the stalk $O_{M,0}$ at 0 is given by $\mathbb{C}[\![t_M]\!]$. 
Let $\circ$ be a ${O}_{M}$-bilinear operation $\circ :{T}_M\times {T}_M \rightarrow {T}_M$. 
On the coordinates $t_M$, the product $\partial_\alpha \circ\partial_\beta$ is written as a linear combination of $\{\partial_\alpha:=\partial/\partial{t^\alpha}\}$ with coefficients in $\mathbb{C}[\![t_M]\!]$. Let $A_{\alpha\beta}{}^\gamma \in\mathbb{C}[\![t_M]\!]$ be a formal power series representing the 3-tensor field such that
\[
	\partial_\alpha\circ \partial_\beta := \sum_\rho A_{\alpha\beta}{}^\rho\cdot\partial_\rho.
\]
%
%
%

The primitive middle-dimensional cohomology 
$H^{n-k}_{\pr}(X_{G})$ or  the cohomology of hypersurface complement $H^{n+k-1}(\BP(\CMcal{E})\setminus X_S)$ (assuming that $X_{G}$ is Calabi-Yau) will be our example of an affine manifold $M$; we will construct an explicit algorithm for a formal flat $F$-manifold structure on such $M$.

\subsection{Universal unfolding, the Euler vector field and the spectrum}\label{sec3.4}
%

Suppose from now on that $X_G\subseteq\mathbf{P}_\Sigma$ is a Calabi-Yau quasi-smooth complete intersection of ample hypersurfaces.
Let 
\be
Y=\Spec (\C[\ud q]_0), \quad M=H^{n+k-1}(\BP(\CMcal{E})\setminus X_S)\simeq \C^\mu, \quad X = Y \times M.
\ee
We recall the notion of a universal unfolding; \cite{Saito} and \cite{ST}. When $\cF$ is a sheaf on a topological space, denote by $\cF_x$ the stalk at $x$. We use a coordinate $\ud t =\{t^\a : \a \in I\}$ for $M$ where $I$ is an index set of cardinality $\mu$.
\begin{definition}[universal unfolding]
A function $\tilde S (\ud q, \ud t)$ on $X$ is called a universal unfolding of $S(\ud q)$ if
\begin{itemize}
\item $\tilde S(\ud q, 0)= S(\ud q)$
\item $\mathrm{a}: \cT_{M,0} \to p_* \CMcal{O}_{X,0}/Jac(\tilde S), \quad X \mapsto \mathrm{a}(X)=\hat X(\tilde S)$ is an $\CMcal{O}_{M,0}$-module isomorphism, where $\hat X$ is a lifting of $X$ to $M$ with respect to $p$ and $Jac(\tilde S) =\langle \frac{\partial \tilde S}{\partial q_1}, \cdots, \frac{\partial \tilde S}{\partial q_{r+k}} \rangle$.
\end{itemize}
\end{definition}

Let $\tilde S$ be a universal unfolding of $S$, which is weight homogeneous with $\wt(\tilde S) =1$.
We consider the following commutative diagram:
\[\begin{tikzcd}
	& { X= Y \times M } \\
	{\C \times M} & M
	\arrow["\varphi"', from=1-2, to=2-1]
	\arrow["p", from=1-2, to=2-2]
	\arrow["\pi"', from=2-1, to=2-2]
\end{tikzcd}\]
where $p$ and $\pi$ are projection maps to the second component and $\phi(\ud q, \ud t) = (\tilde S(\ud q, \ud t), \ud t).$

We consider the Euler vector field $E:= \mathrm{a}^{-1}(\tilde S)$ associated to $\tilde S$.
Due to the weight homogeneity of $\tilde S$, we can choose $E$ by
\bea \label{Euler}
E=\sum_{\a} \wt(t^\a) t^\a \partial_{t^\a}.
\eea

\begin{definition}
Let $(M,\circ, e, E, g)$ be a Frobenius manifold.
Assume that $X \mapsto \nabla^g_X E$ is semi-simple. Then the eigenvalues of this semi-simple map is called the spectrum of $(M,\circ, e, E, g)$.
\end{definition}

%

\subsection{The twisted Gauss-Manin connection on a quantization dGBV algebra}
\label{sec3.5}
We can choose $\Gamma\in\mathcal{A}_0^0[\![\underline{t}]\!]=A_0[\![\underline{t}]\!]$ such that  
\begin{eqnarray*}
\{u_\alpha:=\partial_{t^\alpha}\Gamma|_{\underline{t}=0} : \alpha\in I\} 
\end{eqnarray*}
is a $\mathbb{C}$-basis of $\mathcal{A}_0^0/Q_S(\mathcal{A}_0^{-1})$ and $\tilde S:=S+\Gamma$ is a universal unfolding of $S$. 
We consider the following quantization dGBV algebra:
\[
	\big(\mathcal{A}_0[\![\underline{t}]\!](\!(\hbar)\!),\cdot,Q_{S+\Gamma}, \hbar\Delta+Q_{S+\G},\ell_2^{\Delta}(\cdot,\cdot)\big)
\]
with a formal parameter $\hbar$. 
Let 
\[
		\mathcal{H}_{S+\Gamma} :=\frac{\mathcal{A}_0^0[\![\underline{t}]\!](\!(\hbar)\!)}{(Q_{S+\Gamma}+\hbar\Delta)(\mathcal{A}_0^{-1}[\![\underline{t}]\!](\!(\hbar)\!))},
		\quad\mathcal{H}_{S+\Gamma}^{(m)}:=\frac{\mathcal{A}_0^0[\![\underline{t}]\!][\![\hbar]\!]\hbar^{-m}}{(Q_{S+\Gamma}+\hbar\Delta)(\mathcal{A}_0^{-1}[\![\underline{t}]\!][\![\hbar]\!]\hbar^{-m})}
\]
for $m\in\mathbb{Z}$. 
Then this forms a filtration $\mathcal{H}_{S+\Gamma}^{(m-1)} \subset \mathcal{H}_{S+\Gamma}^{(m)}$ and we have an exact sequence of $O_{M,0}$-modules:
\[
	\begin{tikzcd}
		0 \arrow[r] & \cH_{S+\G}^{(-1)} \arrow[r] &\cH_{S+\G}^{(0)}\arrow[r,"r^{(0)}"] &\mathcal{A}_{0}^0[\![\underline{t}]\!]/Q_{S+\G}(\mathcal{A}_{0}^{-1}[\![\underline{t}]\!])\arrow[r] &0
	\end{tikzcd}
\]
where $r^{(0)}$ is the natural projection map.

\begin{definition}\label{GMC}
	We define \textit{the twisted Gauss-Manin connections} on $\mathcal{H}_{S+\Gamma}$. We use the notation $[\cdot]$ to denote the cohomology class.
	\begin{enumerate}[(i)]
		\item For each $\alpha\in I$, we define a connection $\downtriangle_\alpha^{\frac{S+\Gamma}{\hbar}}:\mathcal{H}_{S+\Gamma}\to \mathcal{H}_{S+\Gamma}$ by
		\begin{align*}
			\downtriangle_\alpha^{\frac{S+\Gamma}{\hbar}}([w])=\downtriangle_{\partial_{t^\a}}^{\frac{S+\Gamma}{\hbar}}([w]):=&\left[e^{\tfrac{-(S+\Gamma)}{\hbar}}\frac{\partial}{\partial t^\alpha}\left(e^{\tfrac{S+\Gamma}{\hbar}}w\right)\right]
			=\left[\frac{\partial}{\partial t^\alpha}w+\frac{1}{\hbar}\frac{\partial(S+\Gamma)}{\partial t^\alpha}w\right]
		\end{align*}
		for $w\in\mathcal{A}_0^0 [\![\underline{t}]\!](\!(\hbar)\!)$.
		\item We also define a connection $\downtriangle_{\hbar^{-1}}^{\frac{S+\Gamma}{\hbar }}:\mathcal{H}_{S+\Gamma}\to \mathcal{H}_{S+\Gamma}$ by
		\begin{align*}
			\downtriangle_{\hbar^{-1}}^{\frac{S+\Gamma}{\hbar}}([w])=\downtriangle_{\partial_{\hbar^{-1}}}^{\frac{S+\Gamma}{\hbar}}([w]):=&\left[e^{\tfrac{-(S+\Gamma)}{\hbar}}\frac{\partial}{\partial \hbar^{-1}}\left(e^{\tfrac{S+\Gamma}{\hbar}}w\right)\right]
			=\left[\frac{\partial}{\partial \hbar^{-1}}w+(S+\Gamma)w\right]
		\end{align*}
		for $w\in\mathcal{A}_0^0 [\![\underline{t}]\!](\!(\hbar)\!)$.
	\end{enumerate}
\end{definition}
These connections are well-defined on cohomology classes, since $Q_{S+\Gamma}+\hbar\Delta$ commutes with $\downtriangle_\alpha^{\frac{S+\Gamma}{\hbar}}$ and $\downtriangle_{\hbar^{-1}}^{\frac{S+\Gamma}{\hbar}}$. Note that 
\[
	\downtriangle_\alpha^{\frac{S+\Gamma}{\hbar}}\left(\mathcal{H}_{S+\Gamma}^{(m)}\right)\subseteq\mathcal{H}_{S+\Gamma}^{(m+1)}\quad\textrm{and}\quad\downtriangle_{\hbar^{-1}}^{\frac{S+\Gamma}{\hbar}}\left(\mathcal{H}_{S+\Gamma}^{(m)}\right)\subseteq\mathcal{H}_{S+\Gamma}^{(m)},\quad m\in\mathbb{Z}.
\]
Thus, $\hbar\downtriangle_\alpha^{\frac{S+\Gamma}{\hbar}}$ preserves the $\hbar$-filtration.
Since $\partial_{\alpha}\Gamma=\partial_{t^\alpha}\Gamma=\hbar\downtriangle_\alpha^{\frac{S+\Gamma}{\hbar}}1$ is a $\mathbb{C}[\![\underline{t}]\!][\![\hbar]\!]$-basis of $\mathcal{H}_{S+\Gamma}^{(0)}$, 
there is a connection matrix $\mathbf{A}_{\alpha\beta}{}^\rho\in\mathbb{C}[\![\underline{t}]\!][\![\hbar]\!]$ with respect to the basis $\{\hbar\downtriangle_\rho^{\frac{S+\Gamma}{\hbar}}1:\rho\in I\}$ such that
\[
	\hbar\downtriangle_\beta^{\frac{S+\Gamma}{\hbar}}(\hbar\downtriangle_\alpha^{\frac{S+\Gamma}{\hbar}}1)=\sum_{\rho\in I}\mathbf{A}_{\alpha\beta}{}^\rho\cdot(\hbar\downtriangle_\rho^{\frac{S+\Gamma}{\hbar}}1)+(Q_{S+\Gamma}+\hbar\Delta)(\mathbf{\Lambda}_{\alpha\beta})
\]
for some $\mathbf{\Lambda}_{\alpha\beta}\in\mathcal{A}_{0}^{-1}[\![\underline{t}]\!][\![\hbar]\!]$ and all $\alpha,\beta\in I$. Suppose that there is $\Gamma$ which makes $\mathbf{A}_{\alpha\beta}{}^\rho$ and $\mathbf{\Lambda}_{\alpha\beta}$ have no $\hbar$-power terms, i.e. there is a connection matrix $A_{\alpha\beta}{}^\rho\in\mathbb{C}[\![\underline{t}]\!]$ such that
\begin{equation}\label{F-QM}
	\hbar\downtriangle_\beta^{\frac{S+\Gamma}{\hbar}}(\hbar\downtriangle_\alpha^{\frac{S+\Gamma}{\hbar}}1)=\sum_{\rho\in I}A_{\alpha\beta}{}^\rho\cdot(\hbar\downtriangle_\rho^{\frac{S+\Gamma}{\hbar}}1)+(Q_{S+\Gamma}+\hbar\Delta)(\Lambda_{\alpha\beta}),
\end{equation}
for some $\Lambda_{\alpha\beta}\in\mathcal{A}_{0}^{-1}[\![\underline{t}]\!]$ and all $\alpha,\beta\in I$. Then, by comparing $\hbar$-power terms, the equation \eqref{F-QM} is equivalent to the following:
\begin{equation}\label{F-QM2}
	\begin{aligned}
		\partial_\alpha\Gamma\cdot\partial_\beta\Gamma&=\sum_{\rho\in I}A_{\alpha\beta}{}^\rho\cdot\partial_\rho\Gamma+Q_{S+\Gamma}(\Lambda_{\alpha\beta}),\\
		\partial_\beta\partial_\alpha\Gamma&=\Delta(\Lambda_{\alpha\beta}),
	\end{aligned}
\end{equation}
for all $\alpha,\beta\in I$. We will use the notations $\G_\a = \partial_\a (\G)$ and $\G_{\a\b}=\partial_\b \partial_\a (\G)$.

\subsection{The first structure connections and a weak primitive form}\label{sec3.6}
We recall the notion of the first structure connections in \cite[Definition 9.6]{Her02}. Consider the pullback diagram
\[\begin{tikzcd}
	{\pi^*TM} & TM \\
	{\C\times M} & M.
	\arrow[from=1-1, to=2-1]
	\arrow[from=1-1, to=1-2]
	\arrow[from=1-2, to=2-2]
	\arrow["\pi", from=2-1, to=2-2]
\end{tikzcd}\]

\begin{definition}[first structure connections]\label{fsc}
Let $(M,\circ, e, \nabla, g, E)$ be a Frobenius manifold satisfying $\Lie_E \circ = \circ$ and $\Lie_E g = D\cdot g$.
For a given $s \in \C$, define the first structure connection on $\pi^*(TM)|_{\C^* \times M}$ by the following formulas
\be
\hat\nabla_X Y&:=& \nabla_X Y + \hbar^{-1} X \circ Y, \\
\hat\nabla_{\partial_{\hbar^{-1}}}^{(s,D)} Y&:=& \nabla_{\partial_{\hbar^{-1}}} Y + E \circ Y+\hbar \left( \nabla_Y E +\left(s+\frac{1}{2}-\frac{D}{2}\right)Y\right)
\ee
for vector fields $X,Y$ on $M$ and a vector field $\partial_{\hbar^{-1}}$ on $\C^*$.
\end{definition}
Let 
\be
J_{S+\G}:=\cA^0_0[[\ud t ]]/Q_{S+\G}(\cA^{-1}_0[[\ud t]]).
\ee
\begin{definition}\label{wpf}
We call $\xi^{(0)}\in \cH_{S+\G}^{(0)}$ a weak primitive form of type $(s, D)$ if the map
\be
\hat \cP_{\xi^{(0)}}: \cT_{M,0} \to \cH_{S+\G}^{(0)}, \quad X \mapsto  \nabla_X^{\frac{{S+\G}}{\hbar}}  \hbar (\xi^{(0)})
\ee
intertwines the twisted Gauss-Manin connection and the first structure connections in Definition \ref{fsc} and $r^{(0)} \circ \hat \cP_{\xi^{(0)}}:\cT_{M,0} \to J_{S+\G}$ is an $\CMcal{O}_{M,0}$-module isomorphism:
\bea\label{feq}
\hat\cP_{\xi^{(0)}}( \hbar \hat\nabla_{X} Y ) = \hbar\nabla_{X}^{\frac{{S+\G}}{\hbar}}  (\hat\cP_{\xi^{(0)}}(Y))
\eea
\bea\label{seq}
\hat\cP_{\xi^{(0)}}( \hat\nabla_{\partial_{\hbar^{-1}}}^{(s,D)}  X ) =  \nabla_{\partial_{\hbar^{-1}}}^{\frac{{S+\G}}{\hbar}} (\hat\cP_{\xi^{(0)}}(X))
\eea
for vector fields $X,Y$ on $M$.
\end{definition}
One can similarly define a formal weak primitive form. Theorem \ref{app} follows from its more precise version below.

\begin{theorem} \label{concreteversion}
Let $X_G\subseteq\mathbf{P}_\Sigma$ be a Calabi-Yau quasi-smooth complete intersection of ample hypersurfaces. 
Then there is a concrete algorithm to find a universal unfolding $S+\G$ of $S$ such that $[1] \in \cH_{S+\G}^{(0)}$ is a formal weak primitive form of type
$(-\frac{n+k+1}{2}, 2-n+k)$.
\end{theorem}
\begin{proof}
It follows from Section \ref{sec3.7} and Section \ref{sec3.8}.
\end{proof}

\begin{corollary}\cite[Proposition 3.5]{Park23}
There is an explicit algorithm to construct formal flat $F$-manifolds $(M, \circ, e, \nabla, E)$ on $M=H^{n+k-1}(\BP(\CMcal{E})\setminus X_S)$ with spectrum\footnote{We obtain $\nabla_{\partial_\a} E= wt(t^\a) \partial_\a$ by using \eqref{Euler} and a $\nabla$-flat coordinate $t^\a$.} $\{\wt(t^\a) : \a \in I\}$.
\end{corollary}

%


\subsection{The first equation of a weak primitive form: an algorithm for formal flat $F$-manifolds}\label{sec3.7}

A system of solutions $(\Gamma, \xi^{(0)})$ to the first equation \eqref{feq} in Definition \ref{wpf} gives rise to a formal flat $F$-manifold structure \cite[Proposition 3.5]{Park23}. We give an algorithm when $[1]\in \cH_{S+\G}^{(0)}$ is a weak primitive form and $\partial_\a =\partial_{t^\a}$ is a flat coordinate with respect to $\nabla$ of a flat $F$-manifold structure. In this case, a simple computation confirms that the system of differential equations \eqref{F-QM2} is equivalent to the first equation \eqref{feq}.

The algorithm will be essentially the same as \cite[section 3]{Park23}, the main difference between the current generalized toric setting and \cite{Park23} (the case of $\BP_\Sigma=\BP^n$) is the grading structure (cohomological grading and $\Cl(\BP(\CMcal{E}))$-grading) of the algebra $\cA$, which leads to the complication of the description of $H^{-1}(\cA, K_S)$; when $\BP_\Sigma=\BP^n$, then \cite[Theorem 1.6, (1.10)]{AS} shows that
\bea \label{minusone}
H^{-1}(\cA, K_S) \cong H^{0}(\cA, K_S) 
\eea
which is crucial to prove so called the weak $Q_S\Delta$-lemma \cite[Lemma 2.8]{Park23}.
The weak $Q_S\Delta$-lemma says that for $f\in\mathcal{A}_{c_B}^{-1}$ such that $Q_S(f)=0$, we have $\Delta(f)=-Q_S(\Delta(v))$ for some $v\in\mathcal{A}_{c_B}^{-2}$.
But in our toric setting, the existence of the isomorphism \eqref{minusone} is not guaranteed.
This phenomenon gives us some flexibility in the algorithm, which we will explain in [Step-$\boldsymbol\ell$ $(\ell\geq2)$] of the algorithm.

In practice, finding explicit $u_\a \in A_0$ using the computer program in the quasi-smooth toric complete intersection case needs more work, since $\cA^0$ has $r-n+1$ gradings $\deg_{c_1}, \cdots, \deg_{c_{r-n}}, \deg_w$; there is only two gradings $\deg_c, \deg_w$ in the case of smooth projective complete intersection.

For reader's convenience, we reproduce an explicit algorithm which calculates $\Gamma, A_{\alpha\beta}{}^\rho$ satisfying the equation \eqref{F-QM} (see \cite{Park23} for a proof why such algorithm holds). We use the following notations:
\begin{equation}\label{Exp_notation}
	\begin{aligned}
		\Gamma&=\sum_{\alpha\in I}u_\alpha\cdot t^\alpha+\sum_{m\geq2}\sum_{\underline{\alpha}\in I^m}\frac{1}{m!}u_{\underline{\alpha}}t^{\underline{\alpha}}\in\mathcal{A}_0^0 [\![\underline{t}]\!],\\
		A_{\alpha\beta}{}^\rho&=a_{\alpha\beta}{}^\rho+\sum_{m\geq1}\sum_{\underline{\alpha}\in I^m}\frac{1}{m!}a_{\alpha\beta\underline{\alpha}}{}^\rho t^{\underline{\alpha}}\in\mathbb{C}[\![\underline{t}]\!],\\
		\Lambda_{\alpha\beta}&=\lambda_{\alpha\beta}+\sum_{m\geq1}\sum_{\underline{\alpha}\in I^m}\frac{1}{m!}\lambda_{\alpha\beta\underline{\alpha}}t^{\underline{\alpha}}\in\mathcal{A}_0^{-1} [\![\underline{t}]\!],
	\end{aligned}
\end{equation}
where $u_{\underline{\alpha}}\in\mathcal{A}_0^0$, $a_{\alpha\beta}{}^\rho,a_{\alpha\beta\underline{\alpha}}{}^\rho\in\mathbb{C}$, $\lambda_{\alpha\beta},\lambda_{\alpha\beta\underline{\alpha}}\in\mathcal{A}_0^{-1}$, and $t^{\alpha_1\cdots\alpha_m}:=t^{\alpha_1}\cdots t^{\alpha_m}$. We are assuming that $u_{\underline{\alpha}},a_{\alpha\beta\underline{\alpha}}{}^\rho,\lambda_{\alpha\beta\underline{\alpha}}$ are invariant under the permutation of indices of $\underline{\alpha}$. In order to make $\{\Gamma_\rho=\partial_\a (\G):\rho\in I\}$ be a $\mathbb{C}[\![\underline{t}]\!][\![\hbar]\!]$-basis of $\mathcal{H}_{S+\Gamma}^{(0)}$, we define the $t^\alpha$-term of $\Gamma$ as $u_\alpha$  where $\{u_\alpha:\alpha\in I\}$ is a $\mathbb{C}$-basis of $\mathcal{A}_0^0/Q_S(\mathcal{A}_0^{-1})$. We have to determine $u_{\underline{\alpha}},a_{\alpha\beta}{}^\rho,a_{\alpha\beta\underline{\alpha}}{}^\rho$ which satisfy the equation \eqref{F-QM}. 
%
\begin{definition}\label{u_alpha}
	Let $|\underline{\alpha}|=m$ which means that $\underline{\alpha}\in I^m$. We define the notation $u_{\underline{\alpha}}^{(i)}$ as follows:
	\[
		u^{(i)}_{\underline{\alpha}}=\sum_{\substack{\underline{\alpha}_1\sqcup\cdots\sqcup \underline{\alpha}_{m-i}=\underline{\alpha}\\\underline{\alpha}_j\neq\emptyset}}\frac{1}{(m-i)!}u_{\underline{\alpha}_1}\cdots u_{\underline{\alpha}_{m-i}},\quad(0\leq i\leq m-1),	
	\]
	where the notation $\underline{\alpha}_1\sqcup\cdots\sqcup \underline{\alpha}_{m-i}=\underline{\alpha}$ means $\underline{\alpha}_1\cup\cdots\cup \underline{\alpha}_{m-i}=\underline{\alpha}$ and $\underline{\alpha}_k\cap \underline{\alpha}_\ell=\emptyset$ for $\ell\neq k$.
\end{definition}
Note that $u_{\underline{\alpha}}^{(i)}$ is invariant under the permutation of indices of $\underline{\alpha}$. For example,
\[
	\begin{aligned}
		u_{\alpha\beta\gamma}^{(0)}&=u_\alpha u_\beta u_\gamma,\\
		u_{\alpha\beta\gamma}^{(1)}&=u_\alpha u_{\beta\gamma}+ u_{\beta} u_{\alpha\gamma}+ u_\gamma u_{\alpha\beta},\\
		u_{\alpha\beta\gamma}^{(2)}&=u_{\alpha\beta\gamma},
	\end{aligned}
	\quad\textrm{and}\quad
	\begin{aligned}
		u_{\alpha\beta\gamma\delta}^{(0)}=&u_\alpha u_\beta u_\gamma u_\delta,\\
		u_{\alpha\beta\gamma\delta}^{(1)}=&u_{\alpha\beta}u_{\gamma}u_{\delta}+u_{\alpha\gamma}u_{\beta}u_{\delta}+u_{\alpha\delta}u_{\beta}u_{\gamma}\\
		&+u_{\beta\gamma}u_{\alpha}u_{\delta}+u_{\beta\delta}u_{\alpha}u_{\gamma}+u_{\gamma\delta}u_{\alpha}u_{\beta},\\
		u_{\alpha\beta\gamma\delta}^{(2)}=&u_{\alpha\beta} u_{\gamma\delta}+ u_{\alpha\gamma} u_{\beta\delta}+ u_{\alpha\delta} u_{\beta\gamma}\\
		&+u_\alpha u_{\beta\gamma\delta}+ u_{\beta} u_{\alpha\gamma\delta}+ u_\gamma u_{\alpha\beta\delta}+ u_{\delta} u_{\alpha\beta\gamma},\\
		u_{\alpha\beta\gamma\delta}^{(3)}=&u_{\alpha\beta\gamma\delta}.
	\end{aligned}
\]\par
Now we give the algorithm for $\Gamma$ and $A_{\alpha\beta}{}^\rho$ as follows:
\begin{description}
	\item[Step-0] Choose a $\mathbb{C}$-basis $\{u_\alpha:\alpha\in I\}$ of $\mathcal{A}_0^0/Q_S(\mathcal{A}_0^{-1})$. 
	There is a partition of $\{u_\alpha:\alpha\in I\}$ in terms of weights\footnote{Since a quasi-smooth complete intersection $X_G$ is a compact $V$-manifold, $\bH=H^{n-k}_{\mathrm{prim}}(X_{G})$ has a pure Hodge structure so that $\bH = \oplus_{p+q=n-k} \bH^{p,q}$. Let $R$ be a map from $\cA_0^0/Q_S(\cA_0^{-1})$ to $\bH$ given by the composition of the map in \eqref{exactprim} and the isomorphism $\cA_0^0/Q_S(\cA_0^{-1}) \simeq H^{n+k-1}(\BP_\Sigma \setminus X_G)$. Then, by \cite[Theorem 3.6]{Mav}, $R(u_\a) \in \bH^{p,q}$ if and only if $\wt(u_\a) =q$.}, i.e. we decompose $\{u_\alpha:\alpha\in I\}=\bigcup_{r=0}^{n-k}\{u_{\alpha}:\alpha\in I_r\}$ where $I=I_0\sqcup\cdots\sqcup I_{n-k}$. 
	\item[Step-1] Determine $a_{\alpha\beta}^{(0)}$ and $\lambda_{\alpha\beta}^{(0)}$ using a basis $\{u_\alpha:\alpha\in I\}$ of $J_S$ as follows:
	\[
		u_\alpha u_\beta=\sum_{\rho\in I}a_{\alpha\beta}^{(0)}{}^\rho u_\rho+Q_S({\lambda_{\alpha\beta}^{(0)}}).
	\]
	Note that $a_{\alpha\beta}^{(0)}{}^\rho\in\mathbb{C}$ is unique and $\lambda_{\alpha\beta}^{(0)}$ is unique up to $\ker Q_S$.  
	Then, define
	$u_{\alpha\beta}:=\Delta(\lambda_{\alpha\beta}^{(0)})$ and $a_{\alpha\beta}{}^\rho:=a_{\alpha\beta}^{(0)}{}^\rho$. 
	\item[Step-$\boldsymbol\ell$ $(\ell\geq2)$] Suppose that $|\underline{\alpha}|=\ell+1$. Determine $a_{\underline{\alpha}}^{(i)}{}^\rho$ and $\lambda_{\underline{\alpha}}^{(i)}$ ($0\leq i\leq \ell-1$) in sequence as follows:
	\begin{equation}\label{ind_u}
			\begin{aligned}
				u_{\underline{\alpha}}^{(0)}&=\sum_{\rho\in I}a_{\underline{\alpha}}^{(0)}{}^\rho u_\rho+Q_S(\lambda_{\underline{\alpha}}^{(0)}),\\
				u_{\underline{\alpha}}^{(1)}-\Delta(\lambda_{\underline{\alpha}}^{(0)})&=\sum_{\rho\in I}a_{\underline{\alpha}}^{(1)}{}^\rho u_\rho+Q_S(\lambda_{\underline{\alpha}}^{(1)}),\\
				&\quad\vdots\\
				u_{\underline{\alpha}}^{(\ell-1)}-\Delta(\lambda_{\underline{\alpha}}^{(\ell-2)})&=\sum_{\rho\in I}a_{\underline{\alpha}}^{(\ell-1)}{}^\rho u_\rho+Q_S(\lambda_{\underline{\alpha}}^{(\ell-1)}).
		\end{aligned}
	\end{equation}
	Then, define $a_{\underline{\alpha}}{}^\rho:=a_{\underline{\alpha}}^{(\ell-1)}{}^\rho$ and\footnote{ If the weak $Q_S\Delta$-lemma holds, then  $\Delta(\lambda_{\underline{\alpha}}^{(i)})$ is unique up to $\mathrm{Im}(Q_S)$, which implies that $a_{\underline{\alpha}}^{(\ell-1)}{}^\rho$ is independent of the choices of $\lambda_{\underline{\alpha}}^{(i)}$. But we do not have the weak $Q_S\Delta$-lemma and so $a_{\underline{\alpha}}^{(\ell-1)}{}^\rho$ depends on the choices of $\lambda_{\underline{\alpha}}^{(i)}$.}
	$u_{\underline{\alpha}}:=\Delta(\lambda_{\underline{\alpha}}^{(\ell-1)})$.
\end{description}
By this inductive algorithm, we can completely determine $\Gamma, A_{\alpha\beta}{}^\rho$, and $\Lambda_{\a\b}$ 


\subsection{The second equation of a weak primitive form: the weight homogeneity of a universal unfolding}
\label{sec3.8}

If a system of solution $(\Gamma, \xi^{(0)})$ to the first equation \eqref{feq} of a weak primitive form is given, then we claim that the weight homogeneity of $S+\G$ implies the second equation \eqref{seq}.
Note that the second equation \eqref{seq} determines the type of the weak primitive form $\xi^{(0)}$.

\begin{proposition}\label{RHBsol}
If $\G$ is a solution to \eqref{F-QM2} satisfying $\wt(\G)=1$ and $\wt(\hbar)=1$, 
then we have
\be
\hat\cP_{[1]}( \hat\nabla_{\partial_{\hbar^{-1}}}^{(-\frac{n+k+1}{2}, 2-n+k)}  \partial_\a ) =  \nabla_{\partial_{\hbar^{-1}}}^{\frac{{S+\G}}{\hbar}} (\hat\cP_{[1]}(\partial_\a)) \quad \text{ in } \cH_{S+\G}^{(0)}.
\ee
\end{proposition}
\begin{proof}

We compute the left-hand side
\be
\hat\cP_{[1]}( \hat\nabla_{\partial_{\hbar^{-1}}}^{(-\frac{n+k+1}{2}, 2-n+k)}  \partial_\a )
&=&\hat\cP_{[1]} \left( \nabla_{\partial_{\hbar^{-1}}} \partial_\a+ E \circ \partial_\a+\hbar \left( \nabla_{\partial_\a} E +\left(-\frac{n+k+1}{2}+\frac{1}{2}-\frac{2-n+k}{2}\right)\partial_\a\right) \right) \\
&=& \hat\cP_{[1]} \biggl( E \circ \partial_\a-\hbar  \left(\wt(\G_\a)+k\right)\partial_\a \biggr)\\
&=& \hbar\nabla_{E\circ \partial_\a}^{\frac{S+\G}{\hbar}}(1) - \hbar \left(\wt(\G_\a)+k\right)\hbar\nabla_{E\circ \partial_\a}^{\frac{S+\G}{\hbar}}(1) 
=\hbar\nabla_{E\circ \partial_\a}^{\frac{S+\G}{\hbar}}(1) - \hbar \left(\wt(\G_\a)+k\right) \G_\a
\ee
where we use the fact that $\nabla_{\partial_\a} E= wt(t^\a) \partial_\a$ for a $\nabla$-flat coordinate $t^\a$.
On the other hand, we have
\be
\nabla_{\partial_{\hbar^{-1}}}^{\frac{{S+\G}}{\hbar}} (\hat\cP_{[1]}(\partial_\a))
=\nabla_{\hbar^{-1}}^{\frac{S+\G}{\hbar}} \hbar \nabla_{\partial_\a}^{\frac{S+\G}{\hbar}} (1) = \nabla_{\hbar^{-1}}^{\frac{S+\G}{\hbar}}\G_\a.
\ee

Let $E_{\wt}:=\sum_{a=1}^{r+k} \wt(q_a)q_a \frac{\partial}{\partial q_a}=\sum_{i=1}^k y_i \frac{\partial}{\partial y_i}$.
We observe
\be
&&\hbar(\hbar\nabla^{\frac{S+\G}{\hbar}}_\hbar + \nabla^{\frac{S+\G}{\hbar}}_E+\nabla^{\frac{S+\G}{\hbar}}_{E_{\wt}} )  \G_\a  \\
&=& \hbar (\hbar \frac{\partial}{\partial \hbar} + E + E_{\wt}) \G_\a + \hbar \left((\hbar \frac{\partial}{\partial \hbar} + E + E_{\wt})(\frac{S+\G}{\hbar})\right) \cdot  \G_\a  \\
&=& \hbar \cdot \wt(\G_\a) \G_\a + \hbar \cdot \wt\left(\frac{S+\G}{\hbar}\right) \cdot \frac{S+\G}{\hbar}\cdot \G_a =\wt(\G_{\a}) \hbar\G_\a,
\ee
where we use $\wt(\Gamma_\a)=1-\wt(t^\a)=\wt(u_\a)$ and $\wt\left(\frac{S+\G}{\hbar}\right)=0$.
Thus
\be
\nabla_{\partial_{\hbar^{-1}}}^{\frac{{S+\G}}{\hbar}} (\hat\cP_{[1]}(\partial_\a)) = \nabla_{\hbar^{-1}}^{\frac{S+\G}{\hbar}}\G_\a&=& -\hbar^2 \nabla^{\frac{S+\G}{\hbar}}_\hbar \G_\a \\
&=&-\hbar\cdot \wt(\G_{\a}) \G_\a+ \hbar \nabla^{\frac{S+\G}{\hbar}}_E \G_\a+ \hbar \nabla^{\frac{S+\G}{\hbar}}_{E_{\wt}} \G_\a.
\ee
Moreover, direct computations show that
\be
&& \hbar \nabla^{\frac{S+\G}{\hbar}}_{E_{\wt}}\G_\a 
=-\hbar k \cdot \G_\a + (Q_{S+\G}+ \hbar \Delta)\left(\G_\a \cdot \sum_{a=1}^{r+k} \wt(q_a)q_a\eta_a\right) 
\ee
and 
\be
 \hbar  \nabla^{\frac{S+\G}{\hbar}}_E\G_\a 
&=& \sum_\b \wt(t^\b)t^\b \left(  \hbar\G_\b \G_\a + \G_{\a\b}\right) \\
&=&\sum_\g \left(\sum_\b \wt(t^\b)t^\b A_{\a\b}{}^{\g}\right)\G_\g +(Q_{S+\G}+\hbar\Delta)\left(\sum_\b \wt(t^\b)t^\b \Lambda_{\a\b}\right)\\
&=&\hbar\nabla_{E\circ \partial_\a}^{\frac{S+\G}{\hbar}}(1) + (Q_{S+\G}+\hbar\Delta)\left(\sum_\b \wt(t^\b)t^\b \Lambda_{\a\b}\right)
\ee
where we use that fact that $\G$ is a solution to \eqref{F-QM2}. This finishes the proof.
%
\end{proof}

\appendix

\section{Generalities on the cohomology of normal varieties}\label{algdRnormal}
Given a normal $\mathbb{C}$-variety $X$ and a coherent $\CMcal{O}_X$-module $\CMcal{F}$, denote
\begin{align*}
\widehat{\CMcal{F}}:=\underline{\mathrm{Mod}}_{\CMcal{O}_X}\left(\underline{\mathrm{Mod}}_{\CMcal{O}_X}(\CMcal{F},\CMcal{O}_X),\CMcal{O}_X\right)
\end{align*}
called the \emph{reflexive hull} of $\CMcal{F}$. Namely, we write
\begin{align*}
\widehat{\Omega}_{X/\mathbb{C}}^\bullet:=\underline{\mathrm{Mod}}_{\CMcal{O}_X}\left(\underline{\mathrm{Mod}}_{\CMcal{O}_X}(\Omega_{X/\mathbb{C}}^\bullet,\CMcal{O}_X),\CMcal{O}_X\right)
\end{align*}
\begin{remark}\label{refhull-coherence} Since $X$ is Noetherian, $\CMcal{O}_X$ is coherent over $\CMcal{O}_X$ so $\underline{\mathrm{Mod}}_{\CMcal{O}_X}(-,\CMcal{O}_X)$ preserves coherence.
\end{remark}
\begin{proposition}\label{reflsheaves} Let $X$ be a normal integral scheme and $\CMcal{F}$ a coherent $\CMcal{O}_X$-module.
\begin{quote}
(1) $\CMcal{F}^\vee:=\underline{\mathrm{Mod}}_{\CMcal{O}_X}(\CMcal{F},\CMcal{O}_X)$ is reflexive, i.e. $(\CMcal{F}^\vee)^{\vee\vee}\cong\CMcal{F}^\vee$.\\
(2) $\CMcal{F}$ is reflexive if and only if $\CMcal{F}$ is torsion free and for each open subset $U\subseteq X$ and each closed subset $Y\subseteq U$ of codimension $\geq2$, ${\CMcal{F}|_U\cong j_\ast\CMcal{F}|_{U\setminus Y}}$ where $j:U\setminus Y\hookrightarrow U$ is an inclusion map.\\
(3) Let $j:U\hookrightarrow X$ be an open subset such that $X\setminus U$ is of codimension $\geq2$. If $\CMcal{F}|_U$ is locally free, then $\widehat{\CMcal{F}}\cong j_\ast\CMcal{F}|_U$.
\end{quote}
\end{proposition}
\begin{proof}
(1) is \cite[Corollary 1.2]{RHSRS}. (2) follows from \cite[Proposition 1.6]{RHSRS}. (3) is \cite[Proposition 8.0.1]{CLS}.
\end{proof}
\begin{corollary}\label{refldR} Let $X$ be a normal $\mathbb{C}$-variety.
\begin{quote}
(1) If $X$ is smooth over $\mathbb{C}$, then $\widehat{\Omega}_{X/\mathbb{C}}^\bullet\cong\Omega_{X/\mathbb{C}}^\bullet$.\\
(2) If $j:\mathrm{Reg}(X)\hookrightarrow X$ is the smooth locus, then $j_\ast\Omega_{\mathrm{Reg}(X)/\mathbb{C}}^\bullet\cong\widehat{\Omega}_{X/\mathbb{C}}^\bullet$
\end{quote}
\end{corollary}
\begin{proof}
(1) If $X/\mathbb{C}$ is smooth, then $\Omega_{X/\mathbb{C}}^\bullet$ is a complex of locally free $\CMcal{O}_X$-modules.\\
\indent (2) Since $\mathbb{C}$ is perfect, the regular locus $\mathrm{Reg}(X)$ of $X$ agrees with the smooth locus so $\mathrm{Reg}(X)\subseteq X$ is open. By Serre's criterion on normality, $X\setminus \mathrm{Reg}(X)$ is of codimension $\geq2$. Since $\Omega_{\mathrm{Reg}(X)/\mathbb{C}}^\bullet$ is a complex of locally free $\CMcal{O}_{\mathrm{Reg}(X)}$-modules, we conclude by Proposition \ref{reflsheaves} (3).
\end{proof}
From this observation, we may define the \emph{algebraic de Rham cohomology} of normal $\mathbb{C}$-varieties to be
\begin{align*}
H_{\widehat{\mathrm{dR}}}^q(X;\mathbb{C}):=\mathbb{H}^q\left(X,\widehat{\Omega}_{X/\mathbb{C}}^\bullet\right)
\end{align*}
which recovers the usual algebraic de Rham cohomology for $X$ smooth over $\mathbb{C}$. If $\varphi:Y\rightarrow X$ is a map of normal $\mathbb{C}$-varieties, then there is an induced $\CMcal{O}_Y$-linear cochain map:
\begin{align*}
\xymatrix{\varphi^\ast\widehat{\Omega}_{X/\mathbb{C}}^\bullet \ar[r] & \widehat{\varphi^\ast\Omega_{X/\mathbb{C}}^\bullet} \ar[r] & \widehat{\Omega}_{Y/\mathbb{C}}^\bullet}
\end{align*}
which gives the desired functoriality for $H_{\widehat{\mathrm{dR}}}^\bullet$.
\begin{proposition}\label{Kunneth} K\"unneth formula holds for $H_{\widehat{\mathrm{dR}}}^\bullet$. More precisely, if $X$ and $Y$ are normal $\mathbb{C}$-varieties, then there is a canonical isomorphism
\begin{align*}
\xymatrix{
H_{\widehat{\mathrm{dR}}}^\bullet(X;\mathbb{C})\otimes_\mathbb{C} H_{\widehat{\mathrm{dR}}}^\bullet(Y;\mathbb{C}) \ar[r]^-\sim & H_{\widehat{\mathrm{dR}}}^\bullet((X\times_\mathbb{C} Y);\mathbb{C})}
\end{align*}
as graded $\mathbb{C}$-vector spaces.
\end{proposition}
\begin{proof}
Given normal $\mathbb{C}$-varieties $X$ and $Y$, form a commutative diagram
\begin{align*}
\xymatrix{
\mathrm{Reg}(X)\times_\mathbb{C}\mathrm{Reg}(Y) \ar[d] \ar[r] & X\times_\mathbb{C}\mathrm{Reg}(Y) \ar[d] \ar[r]^-{\mathrm{pr}_{\mathrm{Reg}(Y)}} & \mathrm{Reg}(Y) \ar[d] \\
\mathrm{Reg}(X)\times_\mathbb{C} Y \ar[d]_-{\mathrm{pr}_{\mathrm{Reg}(X)}} \ar[r] & X\times_\mathbb{C} Y \ar[d]_-{\mathrm{pr}_X} \ar[r]^-{\mathrm{pr}_Y} & Y \ar[d] \\
\mathrm{Reg}(X) \ar[r] & X \ar[r] & \Spec\mathbb{C}
}
\end{align*}
Since $X$ and $Y$ are flat over $\mathbb{C}$, the projections in the above diagram are flat. Since $X$ and $Y$ are Noetherian schemes, $\mathrm{Reg}(X)\hookrightarrow X$ and $\mathrm{Reg}(Y)\hookrightarrow Y$ are quasicompact quasiseparated.
\begin{align*}
&\quad\mathrm{pr}_X^\ast\widehat{\Omega}_{X/\mathbb{C}}^1\\
&\cong\mathrm{pr}_{\mathrm{Reg}(X)}^\ast(\mathrm{Reg}(X)\hookrightarrow X)_\ast\Omega_{\mathrm{Reg}(X)/\mathbb{C}}^1\\
&\cong(\mathrm{Reg}(X)\times_\mathbb{C} Y\hookrightarrow X\times_\mathbb{C} Y)_\ast\mathrm{pr}_{\mathrm{Reg}(X)}^\ast\Omega_{\mathrm{Reg}(X)/\mathbb{C}}^1\\
&\cong(\mathrm{Reg}(X)\times_\mathbb{C} Y\hookrightarrow X\times_\mathbb{C} Y)_\ast(\mathrm{Reg}(X)\times_\mathbb{C}\mathrm{Reg}(Y)\hookrightarrow \mathrm{Reg}(X)\times_\mathbb{C} Y)_\ast\mathrm{pr}_{\mathrm{Reg}(X)}^\ast\Omega_{\mathrm{Reg}(X)/\mathbb{C}}^1\\
&\cong(\mathrm{Reg}(X)\times_\mathbb{C}\mathrm{Reg}(Y)\hookrightarrow X\times_\mathbb{C} Y)_\ast\mathrm{pr}_{\mathrm{Reg}(X)}^\ast\Omega_{\mathrm{Reg}(X)/\mathbb{C}}^1
\end{align*}
where the second isomorphism follows from the flat base change theorem and the third from Proposition \ref{reflsheaves} because $\mathrm{pr}_{\mathrm{Reg}(X)}^\ast\Omega_{\mathrm{Reg}(X)/\mathbb{C}}^1$ is locally free. Similarly,
\begin{align*}
\mathrm{pr}_Y^\ast\widehat{\Omega}_{Y/\mathbb{C}}^1\cong(\mathrm{Reg}(X)\times_\mathbb{C}\mathrm{Reg}(Y)\hookrightarrow X\times_\mathbb{C} Y)_\ast\mathrm{pr}_{\mathrm{Reg}(Y)}^\ast\Omega_{\mathrm{Reg}(Y)/\mathbb{C}}^1
\end{align*}
Combining the above observations, we get
\begin{align*}
&\quad\widehat{\Omega}_{(X\times_\mathbb{C} Y)/\mathbb{C}}^1\\
&\cong(\mathrm{Reg}(X)\times_\mathbb{C}\mathrm{Reg}(Y)\hookrightarrow X\times_\mathbb{C} Y)_\ast\Omega_{(\mathrm{Reg}(X)\times_\mathbb{C}\mathrm{Reg}(Y))/\mathbb{C}}^1\\
&\cong(\mathrm{Reg}(X)\times_\mathbb{C}\mathrm{Reg}(Y)\hookrightarrow X\times_\mathbb{C} Y)_\ast\left(\mathrm{pr}_{\mathrm{Reg}(X)}^\ast\Omega_{\mathrm{Reg}(X)/\mathbb{C}}^1\oplus\mathrm{pr}_{\mathrm{Reg}(Y)}^\ast\Omega_{\mathrm{Reg}(Y)/\mathbb{C}}^1\right)\\
&\cong\mathrm{pr}_X^\ast\widehat{\Omega}_{X/\mathbb{C}}^1\oplus\mathrm{pr}_Y^\ast\widehat{\Omega}_{Y/\mathbb{C}}^1
\end{align*}
which in turn gives
\begin{align*}
\widehat{\Omega}_{(X\times_\mathbb{C} Y)/\mathbb{C}}^\bullet\cong\mathrm{Tot}\left(\mathrm{pr}_X^\ast\widehat{\Omega}_{X/\mathbb{C}}^\bullet\otimes_{\CMcal{O}_{X\times_\mathbb{C} Y}}\mathrm{pr}_Y^\ast\widehat{\Omega}_{Y/\mathbb{C}}^\bullet\right)\cong\mathrm{Tot}\left(\mathrm{pr}_X^{-1}\widehat{\Omega}_{X/\mathbb{C}}^\bullet\otimes_\mathbb{C}\mathrm{pr}_Y^{-1}\widehat{\Omega}_{Y/\mathbb{C}}^\bullet\right)
\end{align*}
Now in the canonical map
\begin{align*}
\xymatrix{\mathbf{R}\Gamma\left(X,\widehat{\Omega}_{X/\mathbb{C}}^\bullet\right)\otimes^\mathbb{L}_\mathbb{C}\mathbf{R}\Gamma\left(Y,\widehat{\Omega}_{Y/\mathbb{C}}^\bullet\right) \ar[r] & \mathbf{R}\Gamma\left(X\times_\mathbb{C} Y,\widehat{\Omega}_{(X\times_\mathbb{C} Y)/\mathbb{C}}^\bullet\right)
}
\end{align*}
the complexes $\widehat{\Omega}_{X/\mathbb{C}}^\bullet$ and $\widehat{\Omega}_{Y/\mathbb{C}}^\bullet$ are bounded and term by term flat over $\mathbb{C}$ so this is an isomorphism by the general construction of cup product of objects in the derived category of ringed space, for example \cite[\href{https://stacks.math.columbia.edu/tag/0G4A}{Tag 0G4A}]{Stacks}.
\end{proof}
\begin{corollary}\label{dR-affine-bundle} Let $S$ be a normal $\mathbb{C}$-variety. If $f:X\rightarrow S$ is an $\mathbf{A}^n$-bundle, then
\begin{align*}
\xymatrix{f^\ast:H_{\widehat{\mathrm{dR}}}^\bullet(S;\mathbb{C}) \ar[r] & H_{\widehat{\mathrm{dR}}}^\bullet(X;\mathbb{C})}
\end{align*}
is an isomorphism.
\end{corollary}
\begin{proof} On a trivializing open subset $U\subseteq S$, there are canonical isomorphisms
\begin{align*}
\xymatrix{H_{\widehat{\mathrm{dR}}}^\bullet\left(f^{-1}(U);\mathbb{C}\right) \ar[r]^-\sim & H_{\widehat{\mathrm{dR}}}^\bullet(U\times_\mathbb{C}\mathbf{A}^n;\mathbb{C}) \ar[r]^-\sim & H_{\widehat{\mathrm{dR}}}^\bullet(U;\mathbb{C})\otimes_\mathbb{C} H_{\widehat{\mathrm{dR}}}^\bullet(\mathbf{A}^n;\mathbb{C}) \ar[r]^-\sim & H_{\widehat{\mathrm{dR}}}^\bullet(U;\mathbb{C})}
\end{align*}
where the second isomorphism comes from Proposition \ref{Kunneth}, and the last from $H_{\widehat{\mathrm{dR}}}^\bullet(\mathbf{A}^n;\mathbb{C})\cong\mathbb{C}[0]$. Hence
\begin{align*}
\xymatrix{f^\ast:H_{\widehat{\mathrm{dR}}}^\bullet(U;\mathbb{C}) \ar[r] & H_{\widehat{\mathrm{dR}}}^\bullet\left(f^{-1}(U);\mathbb{C}\right)}
\end{align*}
is an isomorphism. Choose a finite affine open covering $\CMcal{U}:=\{U_i\}_{i\in I}$ of $S$ consisting of trivializing open subsets, $f^\ast$ induces a map of the associated double complexes:
\begin{align*}
\xymatrix{\displaystyle f^\ast:\check{C}^\bullet\left(\CMcal{U},\widehat{\Omega}_{S/\mathbb{C}}^\bullet\right) \ar[r] & \check{C}^\bullet\left(f^{-1}\CMcal{U},\widehat{\Omega}_{X/\mathbb{C}}^\bullet\right)}
\end{align*}
where $f^{-1}\CMcal{U}:=\{f^{-1}(U_i)\}_{i\in I}$ is a covering of $X$. Then the associated spectral sequence gives
\begin{align*}
\xymatrix{\displaystyle f^\ast:\bigoplus_{i_0,\cdots,i_p\in I}H_{\widehat{\mathrm{dR}}}^q(U_{i_0\cdots i_p};\mathbb{C}) \ar[r] & \displaystyle\bigoplus_{i_0,\cdots,i_p\in I}H_{\widehat{\mathrm{dR}}}^q\left(f^{-1}(U_{i_0\cdots i_p});\mathbb{C}\right)}
\end{align*}
which is an isomorphism by the above observation. Since the double complexes are bounded, the spectral sequence converges to give the desired conclusion.
\end{proof}
\begin{proposition}\label{normal-dR-compare} If $X_\Sigma$ is a normal toric variety, then
\begin{align*}
H_{\widehat{\mathrm{dR}}}^\bullet(X_\Sigma;\mathbb{C})\cong H^\bullet(X_\Sigma,\underline{\mathbb{C}})
\end{align*}
where $\underline{\mathbb{C}}$ is the constant sheaf on $X$.
\end{proposition}
\begin{proof}
In this case,
\begin{align*}
\xymatrix{0 \ar[r] & \underline{\mathbb{C}} \ar[r] & \CMcal{O}_{X_\Sigma} \ar[r]^-d & \widehat{\Omega}_{X_\Sigma/\mathbb{C}}^1 \ar[r]^-d & \widehat{\Omega}_{X_\Sigma/\mathbb{C}}^2 \ar[r]^-d & \cdots}
\end{align*}
is an exact complex by \cite[Lemma 4.5]{VD}, i.e.
\begin{align*}
\xymatrix{(\underline{\mathbb{C}},0) \ar[r] & \left(\widehat{\Omega}_{X_\Sigma/\mathbb{C}}^\bullet,d\right)}
\end{align*}
is a quasi-isomorphism.
\end{proof}
For definitions and notations of the following proposition, see \cite[\href{https://stacks.math.columbia.edu/tag/0162}{Tag 0162}]{Stacks}
or \cite[section 8]{Weibel}.
\begin{proposition}\label{cosimplicial-resolution} Let $A$ be a normal $\mathbb{C}$-algebra. If $\{\Spec B_i\}_{i=1,\cdots,k}$ is an affine open covering of $\Spec A$ (for example $B_i=A_{f_i}$ with $f_1,\cdots,f_k\in A$ generating the unit ideal) then with $B:=B_1\times\cdots\times B_k$, the canonical map into the cosimplicial de Rham complex $\widehat{\Omega}_{B^\bullet/\mathbb{C}}^\bullet$ (usually $B^\bullet$ is denoted by $(B/A)^\bullet$):
\begin{align*}
\xymatrix{\widehat{\Omega}_{A/\mathbb{C}}^\bullet \ar[r] & \widehat{\Omega}_{B^\bullet/\mathbb{C}}^\bullet}
\end{align*}
is a $0$-coskeletal cosimplicial resolution.
\end{proposition}
\begin{proof} Recall that for $n\geq0$,
\begin{align*}
(B/A)^n=B\otimes_AB\otimes_A\cdots\otimes_AB\quad(\textrm{$n+1$ times})
\end{align*}
As $\Spec A$ is affine, for $1\leq i_0,\cdots,i_n\leq k$,
\begin{align*}
\Spec B_{i_0}\cap\cdots\cap\Spec B_{i_n}\cong\Spec B_{i_0}\otimes_A\cdots\otimes_AB_{i_n}
\end{align*}
is still an affine open subset of $\Spec A$. Hence, by the coherence of $\widehat{\Omega}_{A/\mathbb{C}}^\bullet$ (Remark \ref{refhull-coherence})
\begin{align*}
\widehat{\Omega}_{B^\bullet/\mathbb{C}}^\bullet\cong B^\bullet\otimes_A\widehat{\Omega}_{A/\mathbb{C}}^\bullet
\end{align*}
compatible with the restriction maps along the open embeddings $\Spec B_{i_0}\otimes_A\cdots\otimes_AB_{i_n}\hookrightarrow\Spec A$. Being a cosimplicial resolution means that the canonical map induces a quasi-isomorphism:
\begin{align*}
\xymatrix{\widehat{\Omega}_{A/\mathbb{C}}^\bullet \ar[r] & \mathrm{Tot}\left(C^\bullet\widehat{\Omega}_{B^\bullet/\mathbb{C}}^\bullet\right)}
\end{align*}
where $C^\bullet$ takes the unnormalized complex (see \cite[Definition 8.2.1]{Weibel} for example) in the cosimplicial direction, and $\mathrm{Tot}$ takes its total complex. In fact, the total complex will be a \v{C}ech-de Rham complex of $\widehat{\Omega}_{A/\mathbb{C}}^\bullet$ with respect to the covering $\{\Spec B_i\}_{i=1,\cdots,k}$. Therefore, the canonical map above is a quasi-isomorphism by \cite[\href{https://stacks.math.columbia.edu/tag/0FLH}{Tag 0FLH}]{Stacks} because the higher cohomology of each $\widehat{\Omega}_{A/\mathbb{C}}^q$ vanishes on affine open subsets (see \cite[\href{https://stacks.math.columbia.edu/tag/01XB}{Tag 01XB}]{Stacks} for example).
\end{proof}

\section{Reductive group actions and algebraic de Rham complexes}\label{actiondR}
This section is a minor copy of \cite{GJ}. Given an algebraic group $G$ over $\mathbb{C}$ acting on a $\mathbb{C}$-scheme $X$ of finite type, regarding $G$ acting on itself by conjugation, there is a diagram
\begin{align*}
\xymatrix{
G & G\times_\mathbb{C} X \ar[l]_-{\mathrm{pr}_G} \ar[d]^-{\mathrm{pr}_X} \ar@<0.5ex>[r]^-\alpha & X \ar@<0.5ex>[l]^-\sigma \\
& X &
}
\end{align*}
of $G$-equivariant maps where $\alpha$ is the action and $\sigma=(1,-)$ is a section of $\alpha$. This gives a commutative diagram of $G$-equivariant coherent $\CMcal{O}_{G\times_\mathbb{C} X}$-modules
\begin{align*}
\xymatrix{
\alpha^\ast\Omega_{X/\mathbb{C}}^1 \ar[d] \ar[r]^-{d\alpha} & \Omega_{(G\times_\mathbb{C} X)/\mathbb{C}}^1 \ar[d]^-\wr \\
\mathrm{pr}_G^\ast\Omega_{G/\mathbb{C}}^1 & \mathrm{pr}_G^\ast\Omega_{G/\mathbb{C}}^1\oplus\mathrm{pr}_X^\ast\Omega_{X/\mathbb{C}}^1 \ar[l]
}
\end{align*}
Pulling back the left column by $\sigma$, we get an $\CMcal{O}_X$-linear map
\begin{align*}
\xymatrix{
\sigma^\ast\alpha^\ast\Omega_{X/\mathbb{C}}^1 \ar[d]^-\wr \ar[r] & \sigma^\ast\mathrm{pr}_G^\ast\Omega_{G/\mathbb{C}}^1 \ar[d]^\wr \\
\Omega_{X/\mathbb{C}}^1 \ar[r]_-{d\alpha_{X,G}^1} & \mathrm{Lie}(G)^\vee\otimes_\mathbb{C}\CMcal{O}_X
}
\end{align*}
where the dual is taken over $\mathbb{C}$. Using this, we get a $G$-equivariant cochain map
\begin{align*}
\xymatrix{d\alpha_{X,G}:\left(\Omega_{X/\mathbb{C}}^\bullet,d\right) \ar[r] & \left(\mathrm{Lie}(G)^\vee\otimes_\mathbb{C}\Omega_{X/\mathbb{C}}^\bullet,1\otimes_\mathbb{C} d\right)}
\end{align*}
given by
\begin{align*}
d\alpha_{X,G}(df_1\wedge\cdots\wedge df_\ell):=\sum_{i=1}^\ell(-1)^{i-1}d\alpha_{X,G}^1(df_i)\otimes df_1\wedge\cdots\wedge\widehat{df_i}\wedge\cdots\wedge df_\ell
\end{align*}
Now define the complex of \emph{horizontal forms} to be
\begin{align*}
\left(\Omega_{X/\mathbb{C},G}^\bullet,d\right):=\ker\left(\xymatrix{d\alpha_{X,G}:\left(\Omega_{X/\mathbb{C}}^\bullet,d\right) \ar[r] & \left(\mathrm{Lie}(G)^\vee\otimes_\mathbb{C}\Omega_{X/\mathbb{C}}^\bullet,1\otimes_\mathbb{C} d\right)}\right)
\end{align*}
which is still $G$-equivariant.
\begin{example}\label{Euler-Gm} Consider $G=\mathbf{G}_m\cong\Spec\mathbb{C}[\lambda^{\pm1}]$. Then there is an identification
\begin{align*}
\xymatrix{\mathbb{C} \ar[r]^-\sim & \mathrm{Lie}(\mathbf{G}_m) & 1 \ar@{|->}[r] & \displaystyle\lambda\frac{\partial}{\partial\lambda}}
\end{align*}
If $\mathbf{G}_m$ acts on a $\mathbb{C}$-scheme $X$ of finite type, then the above identification gives an $\CMcal{O}_X$-linear map
\begin{align*}
\xymatrixcolsep{3pc}\xymatrix{\theta_w:\Omega_{X/\mathbb{C}}^1 \ar[r]^-{d\alpha_{X,\mathbf{G}_m}^1} & \mathrm{Lie}(\mathbf{G}_m)^\vee\otimes_\mathbb{C}\CMcal{O}_X \ar[r]^-\sim & \CMcal{O}_X}
\end{align*}
which corresponds to a derivation on $X$, called the \emph{Euler derivation} on $X$. If $X$ is affine, then $\Gamma(X,\CMcal{O}_X)$ is naturally graded by the $\mathbf{G}_m$-action: $f\in\Gamma(X,\CMcal{O}_X)$ is said to be of \emph{weight} $w\in\mathbb{Z}$ if $\alpha^\ast f=\lambda^w\mathrm{pr}_X^\ast f$. In this case, if $f\in\Gamma(X,\CMcal{O}_X)$ is homogeneous, then $\theta_w(df)=\deg_w(f)f$.
This Euler derivation can be extended to a degree $-1$ endomorphism of the graded module:
\begin{align*}
\xymatrix{\theta_w:\Omega_{X/\mathbb{C}}^\ell \ar[r] & \Omega_{X/\mathbb{C}}^{\ell-1}}
\end{align*}
given by the rule
\begin{align*}
\xymatrix{df_1\wedge\cdots\wedge df_\ell \ar@{|->}[r] & \displaystyle\sum_{i=1}^\ell(-1)^{i-1}\theta_w(df_i)df_1\wedge\cdots\wedge\widehat{df_i}\wedge\cdots\wedge df_\ell}.
\end{align*}
Then notice that
\begin{eqnarray}\label{exabc}
\Omega_{X/\mathbb{C},\mathbf{G}_m} \cong \Omega_{A[S^{-1}]_{(0)}/\mathbb{C}}, \quad X=\Spec(A[S^{-1}]).
\end{eqnarray}
\end{example}

\begin{proposition}\label{torusdeg} Let $G$ be an algebraic group over $\mathbb{C}$. If $G\times_\mathbb{C}\mathbf{G}_m$ acts on an affine $\mathbb{C}$-scheme $X$ of finite type via
\begin{align*}
\xymatrix{\alpha:G\times_\mathbb{C}\mathbf{G}_m\times_\mathbb{C} X \ar[r] & X}
\end{align*}
then $\theta_w$ preserves the submodule $\Omega_{X/\mathbb{C},G}^\bullet$ and satisfies the following properties:
\begin{quote}
(1) $\theta_w^2=0$.\\
(2) $\theta_w$ is a derivation of the wedge product, i.e. if $\beta$ is an $\ell$-form, then
\begin{align*}
\theta_w(\beta\wedge\gamma)=\theta_w\beta\wedge\gamma+(-1)^\ell\beta\wedge\theta_w\gamma
\end{align*}
(3) $\left(\Omega_{X/\mathbb{C}}^\bullet,d\right)$ is graded by weight as well: $\xi\in\Omega_{X/\mathbb{C}}^\bullet$ is of weight $w\in\mathbb{Z}$ if $\alpha^\ast\xi=\lambda^w\mathrm{pr}_X^\ast\xi$. Moreover, $d$ and $\theta_w$ are weight $0$, i.e. if $\xi$ is homogeneous, then $\theta_w\xi$ and $d\xi$ are homogeneous in $\deg_w$, and $\deg_w\xi=\deg_w\theta_w\xi=\deg_wd\xi$.\\
(4) The wedge product is of weight $0$, i.e. if $\beta$ and $\gamma$ are homogeneous, then $\beta\wedge\gamma$ is homogeneous and $\deg_w(\beta\wedge\gamma)=\deg_w\beta+\deg_w\gamma$.\\
(5) $\ker\theta_w=\Omega_{X/\mathbb{C},\mathbf{G}_m}$.
\end{quote}
If, furthermore, $G$ is reductive, then the following hold.
\begin{quote}
(6) The submodule $(\Omega_{X/\mathbb{C},G}^\bullet)^G\subseteq(\Omega_{X/\mathbb{C}}^\bullet)^G$ is stable under the exterior derivative of $\Omega_{X/\mathbb{C}}^\bullet$.\\
(7) For any homogeneous $\xi$, $(d\theta_w+\theta_wd)\xi=(\deg_w\xi)\xi$
\end{quote}
\end{proposition}
\begin{proof}
This is a combination of \cite[Proposition 2.2]{GJ}, \cite[Proposition 2.3]{GJ}, and \cite[Proposition 2.4]{GJ}
\end{proof}
\begin{proposition}\label{reductive-quotient-de Rham} Let $G$ be a reductive group over $\mathbb{C}$ acting on a smooth affine $\mathbb{C}$-scheme $X$ and $X\rightarrow Y$ the quotient. If $Y$ is normal, then there is an isomorphism
\begin{align*}
\left(\widehat{\Omega}_{Y/\mathbb{C}}^\bullet,d\right)\cong\left(\Omega_{X/\mathbb{C},G}^\bullet,d\right)^G
\end{align*}
\end{proposition}
\begin{proof}
It follows from \cite[Proposition 3.2]{GJ} and \cite[Remark 3.3]{GJ}.
\end{proof}

\end{document}